\documentclass[10pt,twocolumn]{article}

\usepackage{cvpr}
\usepackage{times}
\usepackage{epsfig}
\usepackage{graphicx}
\usepackage{amsmath}
\usepackage{amssymb}
\usepackage[normalem]{ulem}

\usepackage{booktabs} 
\usepackage{nicefrac}  
\usepackage{microtype}
\usepackage[utf8]{inputenc}

\usepackage{xcolor}
\usepackage{wrapfig}
\usepackage{multirow}

\usepackage{amssymb}
\usepackage{amsthm}
\usepackage{bbm}
\usepackage{array}
\usepackage{amsmath}
\usepackage{mathtools}
\usepackage{float}
\usepackage{enumitem}

\usepackage{mathrsfs}
\usepackage{graphicx}
\usepackage{color}

\usepackage{algorithm}
\usepackage{algpseudocode}
\usepackage{thmtools,thm-restate} 

\makeatletter
\def\paragraph{\@startsection{paragraph}{4}%
  \z@\z@{-\fontdimen2\font}%
  {\normalfont\bfseries}}
\makeatother

\usepackage{tikz}

\newcommand{\frechet}{Fr\'{e}chet } 
\newcommand{\fr}{F} 

\newcommand{\R}{\mathbb{R}}

\newcommand{\mc}{\mathcal}

\newcommand{\g}{\gamma}

\newcommand{\w}{\omega}
\newcommand{\s}{\sigma}

\newcounter{desccount}

\newcommand{\descref}[1]{\hyperref[#1]{#1}}

\newcommand{\lp}{\left(}
\newcommand{\rp}{\right)}
\newcommand{\lbar}{\left|}
\newcommand{\rbar}{\right|}
\newcommand{\lnorm}{\left\|}
\newcommand{\rnorm}{\right\|}

\newcommand\mattwo[4]{\left(\begin{smallmatrix}
			{#1} & {#2}\\
     			{#3} & {#4}
                     \end{smallmatrix}\right)}

\newcommand{\borel}{\operatorname{Borel}}

\newcommand{\la}{\langle}
\newcommand{\ra}{\rangle}

\newcommand{\norm}[1]{\|#1\|} 

\newcommand{\Nm}{\mathcal{N}}

\newcommand{\Ngen}{\mathcal{N}}

\newcommand{\dis}{\operatorname{dis}}		

\newcommand{\dn}{d_{\mathcal{N}}}	

\newcommand{\coup}{\mathscr{C}}

\newcommand{\eqN}{[\mathcal{N}]}

\newcommand{\diag}{\operatorname{diag}}
\newcommand{\tr}{\operatorname{tr}}

\newtheorem{theorem}{Theorem}

\newtheorem{proposition}[theorem]{Proposition}
\newtheorem{lemma}[theorem]{Lemma}

 \theoremstyle{definition}
 \newtheorem{example}[theorem]{Example}
\newtheorem{remark}[theorem]{Remark}

\newtheorem{definition}{Definition}
\makeatletter
\newcommand{\pushright}[1]{\ifmeasuring@#1\else\omit\hfill$\displaystyle#1$\fi\ignorespaces}
\newcommand{\pushleft}[1]{\ifmeasuring@#1\else\omit$\displaystyle#1$\hfill\fi\ignorespaces}
\makeatother

\usepackage[pagebackref=true,breaklinks=true,letterpaper=true,colorlinks,bookmarks=false]{hyperref}

\cvprfinalcopy 

\def\cvprPaperID{****} 

\ifcvprfinal\fi
\begin{document}

\title{Gromov-Wasserstein Averaging in a Riemannian Framework}

\author{Samir Chowdhury\\
Stanford University \\
Department of Psychiatry and Behavioral Sciences \\  
{\tt\small samirc@stanford.edu}
\and
Tom Needham\\
Florida State University \\
Department of Mathematics \\
{\tt\small tneedham@fsu.edu}
}

\maketitle

\begin{abstract}
We introduce a theoretical framework for performing statistical tasks---including, but not limited to, averaging and principal component analysis---on the space of (possibly asymmetric) matrices with arbitrary entries and sizes. This is carried out under the lens of the Gromov-Wasserstein (GW) distance, and our methods translate the Riemannian framework of GW distances developed by Sturm into practical, implementable tools for network data analysis. Our methods are illustrated on datasets of letter graphs, asymmetric stochastic blockmodel networks, and planar shapes viewed as metric spaces. On the theoretical front, we supplement the work of Sturm by producing additional results on the tangent structure of this ``space of spaces", as well as on the gradient flow of the Fr\'{e}chet functional on this space.   
\end{abstract}

\section{Introduction}

In a variety of data analysis contexts, one often obtains matrices which are square and asymmetric. Often these matrices arise when studying \emph{networks} \cite{newman2010networks} where the relationships between nodes cannot be measured directly, but have to be inferred from the activity of the nodes themselves. This is the case for biological networks such as the brain, gene regulation pathways, and protein interaction networks.

Inspired by this connectivity paradigm, we refer to arbitrary square matrices as \emph{networks}. The row/column labels are referred to as \emph{nodes}, and the matrix entries are referred to as \emph{edge weights}. Such matrix datasets commonly arise in many other use cases. For example, a practitioner is typically confronted with an $n\times p$ data matrix $X$ where each row is an observation and each column is a variable, from which the covariance matrix is formed. If the dataset is Euclidean, then there is a well-understood duality between the covariance of the variables and the pairwise distances between the observations. More generally, the dataset could be sampled from a Riemannian manifold (or from a distribution whose high density regions live near such a manifold), and the distances between the points could be given by the geodesic distances on the manifold. Even more generally, it may be the case that the data is sampled from a Finsler manifold, and one has access to the quasimetric defined by the asymmetric length structure of the manifold. This may occur when one is sampling data from a dynamical system driven by some potential function: the asymmetry arises because traveling up the potential function is costlier than traveling down \cite{bao2012introduction}. 

In the interest of performing statistics on such data, it is natural to ask how one obtains a \emph{mean} of such matrices. Simply taking a coordinatewise mean does not work in many cases, e.g. when the matrices are of different sizes or are unlabeled. In such situations, one needs to first perform an alignment/registration task that optimally matches the nodes of one network to the nodes of the other. If the matrices are the same size, then the most obvious approach would be to search for an optimal permutation to match nodes between the networks. However, this idea is too restrictive as real-world datasets are frequently of unequal size. Moreover, for large matrices, searching over permutations is prohibitively computationally expensive. For these reasons, one introduces the idea of ``probabilistic matchings". Here, each node is assigned a weight, so that the total weight of the network is one (i.e., a probability measure is assigned to the nodes of the network). Instead of searching over permutations to match nodes between a pair of networks, we can then instead search over the convex set of \emph{couplings} of their probability measures (that is, joint probability distributions whose marginals agree with the original distributions on the input networks). This is the essential idea of \emph{Gromov-Wasserstein distance}, which is defined below.

The goal of this paper is to introduce a theoretical framework for statistical computations on the space of networks. This is achieved by fusing theoretical results on Gromov-Wasserstein distance \cite{sturm2012space}, algorithms for statistics on Riemannian manifolds \cite{pennec}, and recent algorithmic advances for the computation of Gromov-Wasserstein distance \cite{pcs16}. Using this framework, we are able to perform not just averaging, but a plethora of statistical tasks such as principal component analysis and support vector machine classification.

\subsection{Previous Work}

A \emph{metric measure (mm) space} is a compact metric space endowed with a Borel  probability measure. \emph{Gromov-Wasserstein (GW) distance} was first introduced as a metric on the space of all (isomorphism classes of) mm spaces. Theoretical aspects of the GW distance were explored in \cite{dgh-sm,dghlp, sturm2012space}. The work in \cite{dgh-sm,dghlp} was already focused on applications to object matching, while \cite{sturm2012space} explored the Riemannian-like structures induced by GW distance.

In recent years, GW distance has garnered interest in data science communities as a way to compare unlabeled datasets, or datasets containing samples from different ambient spaces. For example, GW distance has been used to explore a variety of network datasets \cite{hendrikson2016using}, as a metric alignment layer in deep learning algorithms for object classification \cite{ezuz2017gwcnn}, to align word embedding spaces for translation applications \cite{alvarez2018gromov}, for several tasks in analysis of large graphs and networks \cite{xu2019gromov,xu2019scalable}, and has been incorporated into generative models across incomparable spaces \cite{bunne2019learning}. Several specialized variants of GW distance have also been recently introduced \cite{memoli2018gromov,titouan2019optimal,titouan2019sliced}.

The problem of computing GW distance was studied from the algorithmic viewpoint in \cite{pcs16}, where a projected gradient descent algorithm was introduced. The main focus was on using GW distance to compute a \emph{Fr\'{e}chet mean with prescribed size} of distance (or kernel) matrices. The main idea of the present paper is to recast the work in \cite{pcs16} using the theoretical Riemannian framework of \cite{sturm2012space} together with statistical algorithms on Riemannian manifolds \cite{pennec}. 
By using this viewpoint, we are able to generalize the work of \cite{pcs16} to a gradient flow that theoretically prescribes the required size for a \frechet mean, while providing a flexible general framework for machine learning tasks on network-valued datasets. This includes the case of asymmetric networks.

An approach similar to that of the present paper to studying statistics on the space of graphs via Riemannian geometry was initiated in \cite{jain2009structure,jain2012learning}. Recently, these ideas were applied to formulate a theory of statistical shape analysis of embedded graphs in \cite{guo2019quotient}. These works perform analysis on graph space by aligning graph nodes over permutations or ``hard matchings", whereas our approach aligns networks via measure couplings or ``soft matchings". The differences between these theories are interesting, and we expect that the correct formalism to use is highly dependent on the particular application.

\subsection{Contributions}

Our specific contributions are as follows. We first provide the gradient of the GW functional on asymmetric networks. This complements a similar result for symmetric matrices in \cite{pcs16}. On the metric geometry side, we provide a concrete exposition of the tangent space structure on this space of asymmetric networks and of the construction of geodesics in the space of networks. This includes the, to our knowledge, first computationally feasible algorithm to produce Sturm geodesics.  We explicitly formulate the iterative \frechet mean algorithm of Pennec as gradient descent of the \frechet functional on the space of networks. The tangent structure provides a framework for vectorizing collections of networks in order to apply standard ML algorithms. We exemplify this by performing averaging and principal component analysis on a database of planar shapes. Our methods can also be used for network compression, and we illustrate this on a toy example of an asymmetric stochastic blockmodel network.

\section{Preliminaries on the GW distance and \frechet means}

\subsection{Networks and the GW distance}\label{sec:networks_GW_distance}

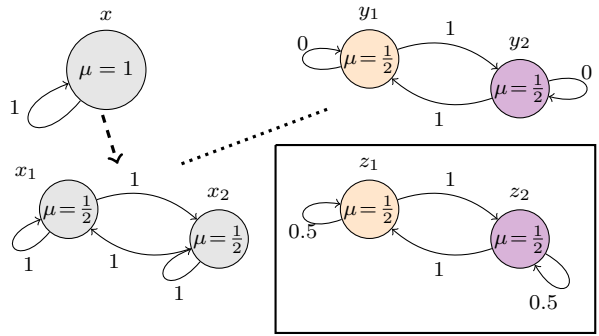
\begin{figure}
\begin{center}
\begin{tikzpicture}[every node/.style={font=\footnotesize},every path/.style={}]

\node[circle,draw,inner sep=0pt,minimum size=30pt,fill=gray!20,label=$x$](1) at (0.5,0.25){$\mu=1$};

\path[->] (1) edge [loop, min distance=10mm,out=230,in=200] node[above left,distance=20pt]{$1$}(1);

\begin{scope}[xshift = 1 cm]
\node[circle,draw,inner sep=0pt,minimum size=17pt,above,fill=orange!20,label=$y_1$](2) at (3,0){$\mu\!=\!\frac{1}{2}$};
\node[circle,draw,inner sep=0pt,minimum size=17pt,fill=violet!30,label=$y_2$](3) at (5,0){ $\mu\!=\!\frac{1}{2}$};
\path[->] (2) edge [loop left, min distance = 7mm] node[above]{$0$}(2);
\path[->] (3) edge [loop right, min distance = 7mm] node[above]{$0$}(3);
\path[->] (2) edge [bend left] node[above]{$1$} (3);
\path[->] (3) edge [bend left] node[below]{$1$} (2);
\end{scope}

\begin{scope}[xshift = 1 cm,yshift=-2cm]
\node[circle,draw,inner sep=0pt,minimum size=17pt,above,fill=orange!20,label=$z_1$](2) at (3,0){$\mu\!=\!\frac{1}{2}$};
\node[circle,draw,inner sep=0pt,minimum size=17pt,fill=violet!30,label=$z_2$](3) at (5,0){ $\mu\!=\!\frac{1}{2}$};
\path[->] (2) edge [loop,out=200,in=170,min distance = 7mm] node[below]{$0.5$}(2);
\path[->] (3) edge [loop,out=330,in=300,min distance = 7mm] node[below left]{$0.5$}(3);
\path[->] (2) edge [bend left] node[above]{$1$} (3);
\path[->] (3) edge [bend left] node[below]{$1$} (2);
\end{scope}

\draw[-,thick] (2.75,-0.75) -- (7,-0.75) -- (7,-3.25) -- (2.75,-3.25) -- (2.75,-0.75);

\draw[->,dashed,very thick] (0.45,-0.35) -- (0.65,-1);
\draw[-,dotted,very thick] (1.5,-1) -- (3.5,-0.25);

\begin{scope}[xshift = -3cm, yshift = -2cm]
\node[circle,draw,inner sep=0pt,minimum size=17pt,above,fill=gray!20,label={135:$x_1$}](2) at (3,0){$\mu\!=\!\frac{1}{2}$};
\node[circle,draw,inner sep=0pt,minimum size=17pt,fill=gray!20,label={$x_2$}](3) at (5,0){$\mu\!=\!\frac{1}{2}$};
\path[->] (2) edge [loop, out=230,in=200,min distance = 7mm] node[below right]{$1$}(2);
\path[->] (3) edge [loop,out=230,in=200, min distance = 7mm] node[below right]{$1$}(3);
\path[->] (2) edge [bend left] node[above left]{$1$} (3);
\path[->] (3) edge [bend left] node[below,pos=0.75]{$1$} (2);

\end{scope}

\end{tikzpicture}
\caption{Networks from Example \ref{ex:simplenets}. \textbf{Top row:} One-node network $X$ and two-node network $Y$. \textbf{Bottom left:} Blown-up form $\hat{X}$ of $X$. Dotted line shows networks to align. \textbf{Inset:} ``average" of $X$ and $Y$.}
\label{fig:simplenets}
\end{center}
\end{figure}

A \emph{measure network} is a triple $(X,\w_X,\mu_X)$ where $X$ is a Polish space (i.e. separable, completely metrizable), $\mu_X$ is a fully supported Borel probability measure, and $\w_X:X\times X \rightarrow \R$ is a square integrable function. The collection of all networks is denoted $\Nm$. When no confusion will arise, we abuse notation and denote the triple $(X,\omega_X,\mu_X)$ by $X$.

The notion of a measure network is quite general and includes several types of spaces which arise in applications. A graph can be represented as a measure network: $X$ is set of $n$ nodes, while $\omega_X$ provides relational information and could be represented by a weighted adjacency matrix, a graph Laplacian, or a matrix of graph distances (for connected graphs or strongly connected digraphs)---see Example \ref{ex:simplenets}. The probability measure $\mu_X$ can be taken to be uniform, or could more generally give higher weight to nodes deemed more important by a particular application. The notion of a measure network is also a strict generalization of that of a metric measure space, as defined in the previous section. For a finite metric measure space, $\omega_X$ can be represented as by its distance matrix. Our definition also includes infinite spaces---this is necessary for certain theoretical completeness results to hold, but the reader more interested in applications can safely restrict their attention to finite spaces.

A \emph{coupling} between two probability measures $\mu_X$ and $\mu_Y$ supported on Polish spaces $X$ and $Y$ is a probability measure on $X\times Y$ with marginals $\mu_X$ and $\mu_Y$. Stated differently, $\mu(A\times Y) = \mu_X(A)$ for all $A \in \borel(X)$ and $\mu(X\times B) = \mu_Y(B)$ for all $B \in \borel(Y)$. The collection of all such couplings is denoted $\coup(\mu_X,\mu_Y)$.

Intuitively, for finite measure networks $X$ and $Y$ and a coupling $\mu$, the value of the coupling $\mu(x_i,y_j)$ can be understood as an assignment of a certain proportion of the mass $\mu_X(x_i)$ to the point $y_j$. In the graph setting, this is a ``soft matching" of the nodes of $X$ and $Y$. One then wishes to find a soft matching which reflects similarity of $X$ and $Y$ as well as possible. This is formalized by the notion of \emph{distortion}, defined below.

Given two networks $(X,\w_X,\mu_X)$ and $(Y,\w_Y,\mu_Y)$ in $\Nm$, one defines the \emph{distortion functional} to be the map 
\begin{align*}
\dis: \coup(\mu_X,\mu_Y) &\rightarrow\R \\
\mu &\mapsto \norm{\w_X - \w_Y}_{L^2(\mu\otimes \mu)}.
\end{align*}
More explicitly, $\dis(\mu)$ is the quantity
$$
\int \lbar \w_X(x,x') - \w_Y(y,y') \rbar^2 \, d\mu(x',y')\,d\mu(x,y),
$$
where the integral is taken over the space $X \times Y \times X \times Y$. The \emph{Gromov-Wasserstein distance} between networks $(X,\w_X,\mu_X)$ and $(Y,\w_Y,\mu_Y)$ is then defined by
$$
\dn(X,Y) := \frac{1}{2} \inf_{\mu \in \coup(\mu_X,\mu_Y)} \dis(\mu).
$$
One may similarly define the $p$-GW distance by taking the $L^p$ norm, but the constructions in this paper rely on the special structure of the $L^2$ case.

The following lemma shows that the infimum in the definition above is always achieved. Minimizers of $\dis$ are referred to as \emph{optimal couplings}. The proof follows directly from \cite[Lemma 1.2]{sturm2012space}.  

\begin{lemma}[Optimality of couplings, \cite{sturm2012space}] 
\label{lem:opt-coup}
Let $(X,\w_X,\mu_X), (Y,\w_Y,\mu_Y) \in \Nm$. Then there always exists a minimizer of $\dis$ in $\coup(\mu_X,\mu_Y)$. 
\end{lemma}

In the finite setting, the notation of \cite{pcs16} admits some useful insights into this minimization problem. We present this notation now. Let $L: \R \times \R \rightarrow \R$ be a loss function. In our case, this will always be defined as $L(a,b) := \lbar a- b\rbar^2$. Next we switch to matrix notation: given a finite space $X = \{x_1,x_2,\ldots, x_n\}$, we write $X_{ik}$ to denote $\w_X(x_i,x_k)$ for $1\leq i,k \leq n$. Suppose $Y$ is another finite space of size $m$. The collection of couplings $\coup(\mu_X,\mu_Y)$ then consists of $n\times m$ matrices $C = (C_{ij})_{ij}$ such that
$$
\sum_i C_{ij} = \mu_Y(y_j) \;\; \mbox{ and } \;\; \sum_j C_{ij} = \mu_X(x_i).
$$
Then one defines the 4-way tensor $\mc{L}(X,Y) := \lp L(X_{ik},Y_{jl}) \rp_{ijkl}$. Given a 4-way tensor $\mc{L}$ and a matrix $\lp C_{ij}\rp_{ij}$, one defines the tensor-matrix multiplication 
\[\mc{L} \otimes C := \lp \sum_{kl} \mc{L}_{ijkl}C_{kl} \rp_{ij}.\]
Next, given two real-valued matrices $A$ and $B$ of the same dimensions, one writes $\la A, B\ra$ to denote the Frobenius inner product $\sum_{ij} A_{ij}B_{ij}$. As observed in \cite{pcs16}, the GW distance between two finite networks $X$ and $Y$ can be written as:
\[ \dn(X,Y) = \frac{1}{2} \min_{C \in \coup(\mu_X,\mu_Y)} \la \mc{L}(X,Y) \otimes C, C \ra^{1/2}.\]
For the reader's convenience, we verify that the dimensions are consistent. If $X$ is an $n$-point space and $Y$ is an $m$-point space. Then $C$ is an $n\times m$ coupling matrix, $\mc{L}(X,Y)$ is an $n^2 \times m^2$ tensor, and the product $\mc{L}(X,Y) \otimes C$ is an $n \times m$ matrix. An alternative matrix formulation \cite{s16} of the term inside the $\min$ is the following:
\begin{align}
    \big(   
    \la \mu_X, X.^{\wedge 2} \mu_X \ra + \la \mu_Y, Y.^{\wedge 2}\mu_Y \ra 
    - 2 \tr(C^TX^TCY)    
    \big)^{\frac{1}{2}}.
    \label{eq:gw-matrix}
\end{align}
Here $.^{\wedge 2}$ denotes the elementwise square. While this formulation is well-known, we provide details in the appendix for the reader's convenience.

\subsection{Weak isomorphism: From transport plans to transport maps via blow-ups}

The space $(\Ngen, \dn)$ is a pseudometric space. Networks $(X,\w_X,\mu_X)$ and $(Y,\w_Y,\mu_Y)$ satisfy $\dn(X,Y) = 0$ if and only if there exists a Borel probability space $(Z,\mu_Z)$ with maps $\pi_X:Z \rightarrow X$ and $\pi_Y : Z \rightarrow Y$ such that the pushforward measures satisfy  $(\pi_X)_\#\mu_Z = \mu_X$, $(\pi_Y)_\#\mu_Z = \mu_Y$, and the pullbacks $(\pi_X)^*\w_X$, $(\pi_Y)^*\w_Y$ satisfy $\norm{ (\pi_X)^*\w_X - (\pi_Y)^*\w_Y}_\infty = 0$, where $(\pi_X)^\ast \w_X(z,z'):=\w_X(\pi_X(z),\pi_X(z'))$.  In this case, $X$ and $Y$ are said to be \emph{weakly isomorphic} and write $X \cong^w Y$. The space $Z$ is referred to as a \emph{common expansion} of $X$ and $Y$. We write $[X] = [X,\w_X,\mu_X]$ to denote the weak isomorphism class of $X = (X,\w_X,\mu_X)$ in $\Ngen$. The collection of equivalence of classes of networks will be denoted $\eqN$. See also \cite{gwnets} for more details on weak isomorphism.

In the case of finite networks, weak isomorphism is especially useful as it allows one to convert transport plans (i.e., couplings between the networks' measures) to transport maps (i.e., measure-preserving maps between networks). This observation plays a key role in both our theory and algorithms. We present this construction next.

\begin{definition}[Blowups]
\label{def:blowup}
Let $X, Y$ be finite networks, and let $\mu \in \coup(\mu_X,\mu_Y)$. Let $\mathbf{u}:= (u_x)_{x\in X}$ be a vector where $u_x := |\{y \in Y : \mu(x,y) > 0\}|$. Also define $\mathbf{v}:= (v_y)_{y\in Y}$ by setting $v_y := |\{x \in X : \mu(x,y) > 0\}|$. Next define $X[\mathbf{u}]$ to be the node set $\bigcup_{x\in X}\{ (x,i) : 1 \leq i \leq u_x\} $. Fix $x \in X$ and let $y_1,y_2,\ldots, y_{u_x}$ denote the $y \in Y$ such that $\mu(x,y) >0$. Define $\mu_{X[\mathbf{u}]}((x,i)) := \mu(x,y_i)$. Finally, for $x,x' \in X$ and $1\leq i \leq u_x$, $1\leq j \leq u_{x'}$, define $\w_{X[\mathbf{u}]}((x,i),(x',j)) = \w_X(x,x')$. Similarly define $(Y[\mathbf{v}], \w_{Y[\mathbf{v}]}, \mu_{Y[\mathbf{v}]})$. The crux of this construction is that while $X[\mathbf{u}], Y[\mathbf{v}]$ are weakly isomorphic to $X$ and $Y$, respectively, the initial transport plan $\mu$ naturally expands to a transport map from $X[\mathbf{u}]$ to $Y[\mathbf{v}]$. We refer to the process of constructing $X[\mathbf{u}]$ from $X$ as a \emph{blow-up}.
\end{definition}

\begin{definition}[Alignment]\label{def:alignment}
Let $X, Y$ be finite networks on $n$ and $m$ nodes, respectively, and let $\mu$ be an optimal coupling. 
We refer to the $n \times m$ binary matrix $\mathbf{1}_{\mu > 0}$ as the \emph{binarization} of $\mu$: this matrix has the same dimensions as $\mu$, has a 1 where $\mu >0$, and $0$ elsewhere. 
By taking appropriate blow-ups, we obtain (possibly enlarged) networks $\hat{X}$, $\hat{Y}$ and an optimal coupling $\hat{\mu}$ such that the binarization $\mathbf{1}_{\hat{\mu} > 0}$ of $\hat{\mu}$ is a permutation matrix. 
Then we may \emph{align} $\hat{Y}$ to $\hat{X}$ by defining $ \hat{Y} \leftarrow \mathbf{1}_{\hat{\mu} > 0} \hat{Y} \mathbf{1}_{\hat{\mu} > 0}^T$. 
The corresponding realignment of the optimal coupling is given by $\hat{\mu} \leftarrow \mathbf{1}_{\hat{\mu} > 0}\hat{\mu}$. 
Note that we then have $\hat{\mu} = \diag(\mu_{\hat{X}})$.
We refer to this process of blowing up and realigning as \emph{aligning $Y$ to $X$}. 
After aligning, $\dn(X,Y)$ is given by
\[ \sum_{i, j = 1}^n |\w_{\hat{X}}(x_i,x_j) - \w_{\hat{Y}}(y_i,y_j)|^2 \mu_{\hat{X}}(x_i)\mu_{\hat{X}}(x_j).\]
\end{definition}

\begin{example}[Blowing up and aligning simple networks]
\label{ex:simplenets}
Consider the weighted networks $X$ and $Y$ shown in the top row of Figure \ref{fig:simplenets}. We represent these as measure networks by taking $\w_X$ and $\w_Y$ to be weighted adjacency matrices. Concretely, let $X = \{x\}$, $\w_X = (1)$, and $\mu_X(x) = 1$. Also let $Y = \{y_1,y_2\}$, $\w_Y = \mattwo{0}{1}{1}{0}$, and $\mu_Y(y_1) = \mu_Y(y_2) = 1/2$. Since $X$ is a one-node network, the unique coupling $\mu \in \coup(\mu_X,\mu_Y)$ is given by $\mu(x,y_1) = \mu(x,y_2) = 0.5$. To convert $\mu$ to a transport map, $X$ is blown-up to $\hat{X} = \{x_1,x_2\}$, with $\omega_{\hat{X}} =\mattwo{1}{1}{1}{1}$ and $\mu_{\hat{X}}(x_1) = \mu_{\hat{X}}(x_2) =  1/2$, whence $\hat{\mu} = \mattwo{0.5}{0}{0}{0.5}$. Intuitively, the average of $\hat{X}$ and $Y$ should be the network $Z = \{z_1,z_2\}$, $\w_Z = \mattwo{0.5}{1}{1}{0.5}$, and $\mu_Z(z_1) = \mu_Z(z_2) = 0.5$. This intuition will be formalized below.

\end{example}

\subsection{Computing GW distance}

It was implicitly observed in Section \ref{sec:networks_GW_distance} that for finite measure networks $X$ of size $n$ an $Y$ of size $m$, the squared distortion of a coupling matrix $C \in \coup(\mu_X,\mu_Y) \subset \R^{n \times m}$ is given by
$$
\dis(C)^2 =  \left<\mathcal{L}(X,Y)\otimes C, C\right>.
$$
For fixed $X$ and $Y$, let $A_{XY}$ denote the linear map from $\R^{n \times m}$ to itself given by $A_{XY}C := \mathcal{L}(X,Y)\otimes C$. The GW optimization problem seeks a minimizer of the map
\[
C \mapsto \left<A_{XY}C,C\right>
\]
over the convex polytope $\coup(\mu_X,\mu_Y)$ and is therefore an instance of a \emph{quadratic programming} problem. 

Following \cite{pcs16}, we approximate GW distance by finding local minimizers for the GW optimization problem via projected gradient descent. Since we are allowing asymmetric weight functions, the linear map $A_{XY}$ may be asymmetric. This distinguishing feature from the setting of \cite{pcs16} must be accounted for when computing the gradient.

\begin{proposition}\label{prop:gradient_of_GW}
The gradient of $C \mapsto \left<A_{XY}C,C\right>$ is given by 
$$
\lp A_{XY} + A_{XY}^\ast \rp C,
$$
with $A_{XY}^\ast$ denoting the adjoint of $A_{XY}$. 
\end{proposition}

\subsection{\frechet means}

Given a collection of networks $S = \{ X_1,X_2,\ldots, X_n \}$, a \emph{\frechet mean} of $S$ is a minimizer of the functional 
\[\fr_S(Z) := \frac{1}{n}\sum_{i = 1}^n \dn(X_i,Z)^2.\]

In \cite{pcs16}, the approach for calculating a \frechet mean was as follows: (1) fix a cardinality $N$ for the target space $Z$, (2) minimize over choices of $\w_Z$, and (3) optimize over couplings $C_i \in \coup(\mu_{X_i},\mu_Z)$. The last two steps are repeated until convergence. 

In contrast, the scheme we present in this paper follows ideas of Pennec on averaging in (finite-dimensional) Riemannian manifolds \cite{pennec} coupled with the work of Sturm on developing the Riemannian structure of generalizations of metric measure spaces. Informally, the idea is as follows: start with a ``seed" network $X$, use log maps to lift geodesics $X \rightarrow X_i$ to vectors in the tangent space at $X$, average the vectors, use an exponential map to map down to $\Ngen$, and iterate this procedure until convergence. Theoretically, we justify this procedure by showing that it agrees with the downward gradient flow of the \frechet functional.

\section{Metric geometry of $\eqN$}

Given a network $(X,\w_X,\mu_X)$, the only requirement on $\w_X$ is that it needs to be square integrable; i.e., we need $\w_X \in L^2(X^2, \mu_X^{\otimes 2})$. We will show that this flexibility allows us to define structures on $\eqN$ analogous to those of a Riemannian manifold, such as geodesics, tangent spaces and exponential maps. These structures are defined using language from the theory of analysis on metric spaces. In our setting, they can be defined in a surprisingly concrete way, allowing us to sidestep the need to invoke any deep concepts or results---see \cite{burago2001course,ambrosio2008gradient} for general introductions to the theory.

\subsection{Geodesics}

In the metric geometry sense, a \emph{geodesic} from $[X]$ to $[Y]$ in $\eqN$ is a continuous map $\g$ from a closed interval $[S,T]$ into $\eqN$ satisfying the property
\begin{equation}\label{eqn:geodesic_property}
\dn(\g(s),\g(t)) = \frac{|t-s|}{T-S} \cdot \dn([X],[Y])
\end{equation}
for all $s,t \in [S,T]$. The geodesic is \emph{unit speed} if $T-S = \dn([X],[Y])$. We can of course assume without loss of generality that our domain interval is always of the form $[0,T]$. We say that a geodesic $\gamma:[0,T] \rightarrow \eqN$ \emph{emanates from $[X]$} if $\gamma(0) = [X]$. 

It follows from work in \cite{sturm2012space} that geodesics can always be constructed in $\eqN$ (although they need not be unique).
Let $X,Y \in \Nm$, and let $\mu$ be an optimal coupling (cf. Lemma \ref{lem:opt-coup}). For each $t \in [0,1]$, define 
\begin{equation}\label{eqn:geodesic}
\g(t):= [X\times Y, \Omega_t, \mu],
\end{equation}
where 
\begin{align*} 
&\Omega_t((x,y),(x',y')):= (1-t)\,\w_X(x,x') + t\, \w_Y(y,y').
\end{align*}
It is easy to see that $\g(0) = [X]$ and $\g(1) = [Y]$. A relatively straightforward computation then shows that $\g$ satisfies \eqref{eqn:geodesic_property} (cf. \cite[Theorem 3.1]{sturm2012space}). Note that the underlying set of this geodesic is always $X\times Y$. In particular, $\g(0)$ is the triple $[X\times Y, \w_X, \mu]$ where $\w_X$ is defined (by abuse of notation) on $X\times Y \times X \times Y$ as $ \w_X \lp (x,y),(x',y')\rp := \w_X(x,x').$ Here the $\w_X$ on the right hand side is the original function defined on $X \times X$.

For networks $X$ and $Y$ of sizes $m$ and $n$, respectively, a naive implementation of the geodesic described above is represented by a measure network of size $m \cdot n$. We later consider iterative algorithms where this size blowup would quickly become intractable. Instead, we offer the following minimal size representation of a geodesic. We compute an optimal coupling of $X$ and $Y$, then blow up and align the networks as in Definitions \ref{def:blowup} and \ref{def:alignment}; an example for networks coming from simple graphs is shown in Figure \ref{fig:splittingAndMatching}. The geodesic described above is then represented (up to weak isomorphism) by interpolating the blown up and aligned weight matrices for $X$ and $Y$; such a geodesic for graph networks is shown in Figure \ref{fig:gromov_wasserstein_geodesic}. We have observed empirically that such a minimal size geodesic is typically represented as a path of matrices of size proportional to $m + n$, rather than the size $m \cdot n$ naive representation. Experimental evidence of this size reduction is provided in the supplementary materials, Section \ref{sec:support_sizes}.

\begin{figure}
    \centering
    \includegraphics[width=0.4\textwidth]{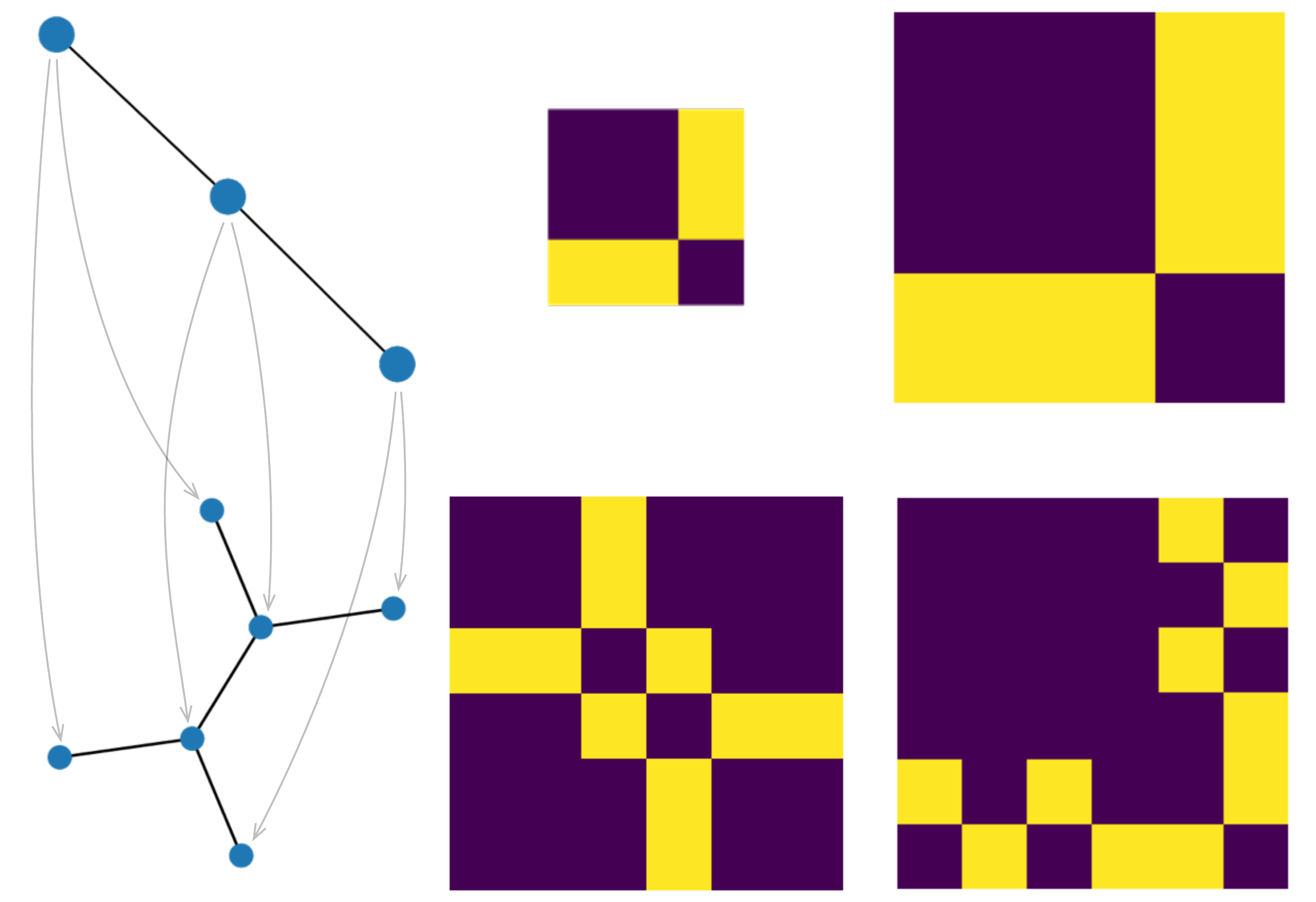}
    \caption{\textbf{Left column:} Two graphs to be matched. Each graph is a measure network with all edge weights equal to one and uniform node weights. The lighter arrows depict an optimal coupling between the graphs; the mass from each node from the first graph is distributed evenly to two nodes in the second graph as indicated.
    \textbf{Middle column:} Graphs are represented as measure networks by their adjacency matrices of size $3 \times 3$ and $6 \times 6$, respectively. 
    \textbf{Right column:} Matrix representations of the measure networks after they have been blown up and aligned according to the optimal coupling. Resulting matrices are both size $6 \times 6$. The matrix for the first graph is doubled in size; at a graph level, copies of each node are created. The matrix for the second graph has remained the same size but has been permuted; this is just a relabelling of the nodes based on how they are matched with nodes of the first graph.}
    \label{fig:splittingAndMatching}
\end{figure}

\begin{figure*}
    \centering
    \includegraphics[width = 0.9 \textwidth]{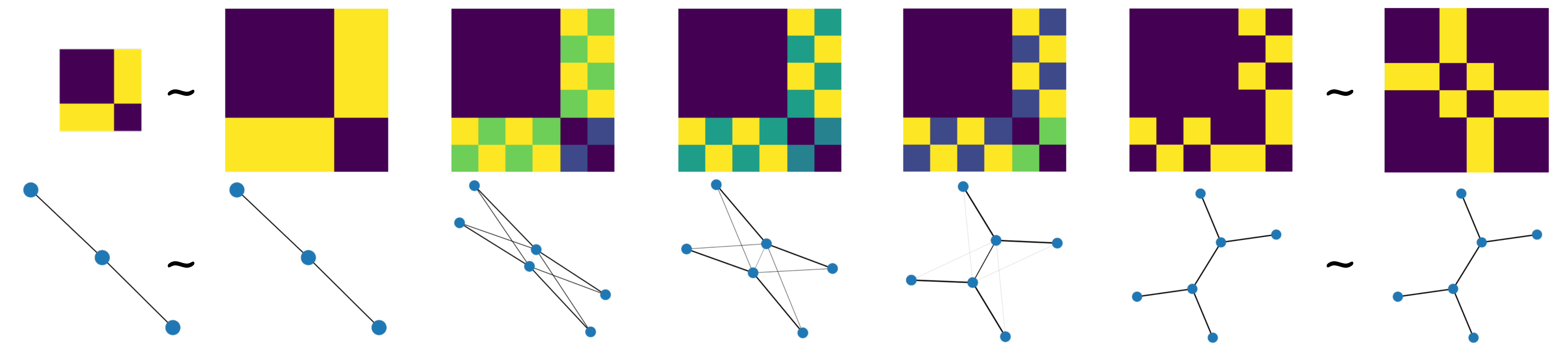}
    \caption{Geodesic between the graphs of Figure \ref{fig:splittingAndMatching}. The adjacency matrices are blown up and realigned. The blown up graph second from the left is drawn with split nodes superimposed. The geodesic interpolates matched edge weights. At the graph level, the geodesic is illustrated with node sizes corresponding to node weights and edge thickness corresponding to edge weights.}
    \label{fig:gromov_wasserstein_geodesic}
\end{figure*}

\subsection{Tangent space}

For a point $[X] \in \eqN$, we define the \emph{tangent space} to be:
\[
T_{[X]} := \bigcup_{Z \in [X]} L^2 \lp Z^2,\mu_Z^{\otimes 2}\rp/\sim,
\]
where $\sim$ is defined as follows. For $\lp Y,\omega_Y,\mu_Y\rp$ and $\lp Z, \omega_Z, \mu_Z \rp$ in $[X]$ and  functions $f \in L^2 \lp Y^2,\mu_Y^{\otimes 2}\rp$ and $g \in L^2\lp Z^2,\mu_Z^{\otimes Z}\rp$, we declare $f \sim g$ if and only if there exists a coupling $\mu$ of $\mu_Y$ and $\mu_Z$ such that 
$$
\omega_Y(y,y') = \omega_Z(z,z') \;\; \mbox{ and } \;\; f(y,y') = g(z,z')
$$
for $\mu^{\otimes 2}$-a.e. $\lp (y,z), (y',z') \rp \in (Y \times Z)^2$. Elements of $T_{[X]}$ are called \emph{tangent vectors to $[X]$} and are denoted $[f]$, where $f$ is an $L^2$ function defined on some representative of $[X]$.

We define an inner product $\left<\cdot,\cdot\right>_{[X]}$ on each tangent space $T_{[X]}$ as follows: for  $f \in L^2(Y^2,\mu_Y^{\otimes 2})$ and $g \in L^2((Y')^2, \mu_Y^{\otimes 2})$ with $Y,Y' \in [X]$,
\begin{align*}
&\left<[f],[g] \right>_{[X]} := \left<(\pi_Y)^\ast f, (\pi_{Y'})^\ast g \right>_{L^2(Z^2,\mu_Z^{\otimes 2})},
\end{align*}
where $Z$ is any measure network realizing the `tripod' in the definition of weak isomorphism between $Y$ and $Y'$. One can check that this value does not depend on any of the choices made and therefore gives a well-defined inner product. The norm induced by this inner product is denoted $\|\cdot\|_{[X]}$. One can check that it reduces to the formula
$$
\|[f]\|_{[X]} := \|f\|_{L^2(Y^2,\mu_Y^{\otimes 2})},
$$
where $f \in L^2(Y^2,\mu_Y^{\otimes 2})$ and $Y \in [X]$. 

\begin{remark}
The tangent space $T_{[X]}$ is not a bona fide vector space, but is a vector space quotiented out by a ``symmetry group" consisting of optimal self-couplings of a certain minimal representative of $X$ \cite[Section 6]{sturm2012space}. When $[X]$ has a representative with only the trivial symmetry, i.e. the diagonal coupling, the tangent space is a Hilbert space. In particular, this phenomenon endows $[\Ngen]$ with the structure of a Riemannian orbifold. 
\end{remark}

\subsection{Exponential map}

We define the \emph{exponential map at $[X]$},
$$
\exp_{[X]}: T_{[X]} \rightarrow \eqN,
$$
as follows. For $f \in L^2 \lp Z,\mu_Z^{\otimes 2}\rp$, with $Z \in [X]$, let
\[
\exp_{[X]}([f]) := [Z, \omega_Z + f, \mu_Z].
\]
After unwrapping the various notions of equivalence involved, one is able to show that this map is well-defined. This map is analogous to the exponential map in a Riemannian manifold. We demonstrate this concretely in the finite setting with the next proposition.

\begin{proposition}\label{prop:tangent_matrices}
Let $X$ be an finite measure network. There exists  $\epsilon_{[X]} > 0$ such that for any tangent vector represented by $f \in L^2(Z^2,\mu_Z^{\otimes 2})$ with $Z \in [X]$ satisfying $|f(z,z')| < \epsilon_{[X]}$ for all $(z,z') \in Z \times Z$, $\exp_{[X]}([f])$ is the endpoint of a geodesic emanating from $[X]$.

\end{proposition}

\subsection{Log map}

We wish to show that $\exp_{[X]}$ has a local inverse, called the \emph{log map at $[X]$}. Let $Y$ be a finite measure network and let $\mu$ be an optimal coupling of $X$ and $Y$. Define the \emph{log map with respect to $\mu$} as follows. Use $\mu$ to expand and align the measure networks to 
\begin{equation}\label{eqn:aligned_spaces}
\hat{X} = \lp \hat{X},\omega_{\hat{X}},\mu_{\hat{X}} \rp \;\; \mbox{ and } \;\; \hat{Y} = \lp \hat{X}, \omega_{\hat{Y}}, \mu_{\hat{X}} \rp
\end{equation}
so that the identity map on the set $\hat{X}$ induces an optimal coupling of $\hat{X}$ with $\hat{Y}$. We then define
\begin{equation}\label{eqn:log_map}
\log_{[X]}^\mu([Y]) := [\omega_{\hat{Y}} - \omega_{\hat{X}}].
\end{equation}
It immediately follows that
$$
\exp_{[X]}\lp \log_{[X]}^\mu ([Y]) \rp = [Y].
$$
This provides a surjectivity result for the exponential map. On the other hand, the following lemma provides an injectivity result. Its proof is similar to that of Proposition \ref{prop:tangent_matrices}.

\begin{lemma}\label{lem:bijective_lemma}
Let $X$ be a finite measure network. The exponential map $\exp_{[X]}$ is injective on the set of $[f]$ with $f \in L^2(Z^2,\mu_Z^{\otimes 2})$ such that $Z$ is finite and $f$ satisfies $|f(z,z')| < \epsilon_{[X]}/2$ for all $z,z' \in Z$. 
\end{lemma}

We now define the \emph{log map at $[X]$}, $\log_{[X]}$, to be the local inverse of $\exp_{[X]}$ on finite measure networks near $[X]$. For a finite measure network $Y$, we define $\log_{[X]}([Y]) = \log_{[X]}^\mu([Y])$ as in \eqref{eqn:log_map}, where $\mu$ is any optimal coupling of $X$ with $Y$. The lemma then provides a certificate to check that the image of the log map did not depend on a choice of optimal coupling.

\subsection{Gradients}

Let $F:\eqN \rightarrow \R$ be a functional, let $[X] \in \eqN$ and let $[f] \in T_{[X]}$. Define the \emph{directional derivative of $F$ at $[X]$ in the direction $[f]$} to be the limit
$$
D_{[f]} F ([X])  := \lim_{t \rightarrow 0^+} \frac{1}{t} \lp F\lp \exp_{[X]}([t \cdot f]) \rp - F([X]) \rp,
$$
provided it exists. We say that $F$ is \emph{differentiable} at $[X]$ if all directional derivatives exist. 

For a differentiable functional $F$, a \emph{gradient} of $F$ at $[X]$ is a tangent vector $\nabla F([X])$ satisfying
$$
D_{[f]} F ([X]) = \left<[f],\nabla F([X]) \right>_{[X]}
$$
for all $[f] \in T_{[X]}$. 

The next lemma follows from \cite[Lemma 6.24]{sturm2012space}.

\begin{lemma}
Let $F$ be a differentiable functional. If the gradient of $F$ at $[X]$ exists, then it is unique and satisfies
$$
\left\| \nabla F([X]) \right\|_{[X]} = \sup \left\{ D_{[f]} F ([X]) \mid \|[f]\|_{[X]} = 1 \right\}.
$$
\end{lemma}

We are now able to derive an explicit expression for the gradient of the \frechet functional for finite networks.

\begin{proposition}\label{prop:gradient_frechet_functional}
Let $S= \{Y_1,Y_2,\ldots, Y_n\}$ be a collection of finite networks, and let $X$ be another finite network. Suppose each $Y_i$ has been aligned to $X$, so that each of $\{X,Y_1,\ldots, Y_n\}$ has $m$ nodes. Then the gradient of the \frechet functional $F_S$ at $[X]$ is represented by the $m\times m$ matrix $\nabla \fr_S(X)$ defined by
\[ \lp\nabla\fr_S(X)\rp_{ij} = 2 \lp \w_X(x_i,x_j) - \frac{1}{n}\sum_{k=1}^n \w_{Y_k}(y_i,y_j) \rp.\]
\end{proposition}

\begin{remark} The preceding proposition gives us a meaningful description of a \frechet mean. Specifically, let $S = \{Y_1,\ldots, Y_n\}$ be a collection of finite networks, and let $X$ be such that $\nabla \fr_S(X) = 0$. Suppose also that $X$ is aligned to each $Y_i$. Then $X$ has the property that $\w_X(x_i,x_j) = \frac{1}{n}\sum_{k=1}^n \w_{Y_k}(y_i,y_j)$ for each $x_i,x_j \in X$.
In other words, $\w_X$ comprises arithmetic means of entries in the $\w_{Y_k}$. 
\end{remark}

\section{Experiments}

We now provide details of our computational experiments. Algorithms and numerous empirical remarks are provided in the supplementary materials. Python implementations are available on GitHub \cite{code}. Our code makes heavy use of the Python Optimal Transport Library \cite{flamary2017pot}.

\subsection{Geodesic Examples}\label{sec:geodesic_examples}

\begin{figure}
    \centering
    \includegraphics[width = 0.45\textwidth]{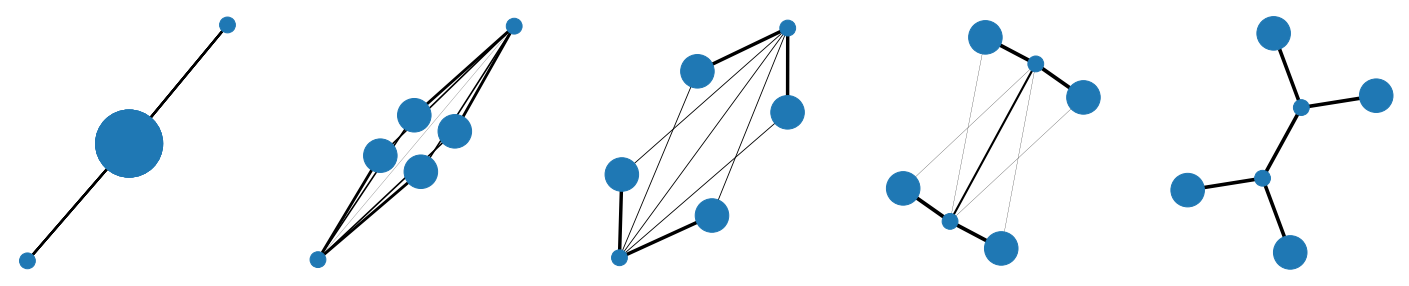}
    \includegraphics[width = 0.45\textwidth]{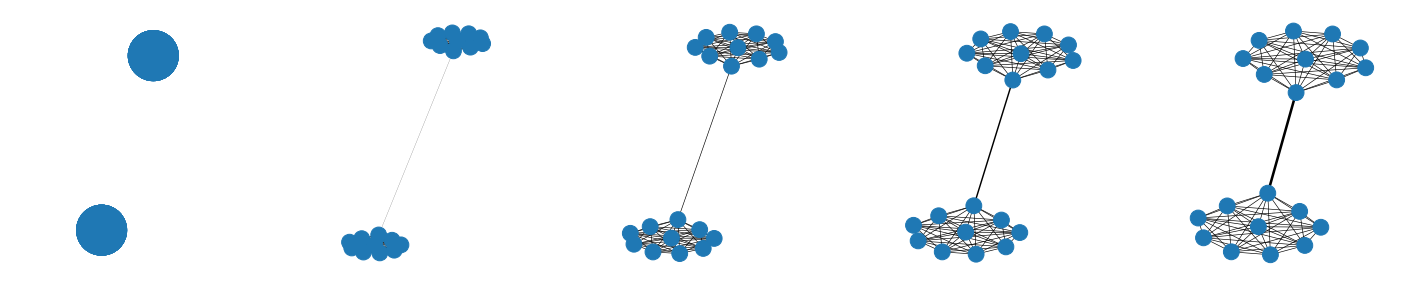}
    \includegraphics[width = 0.45\textwidth]{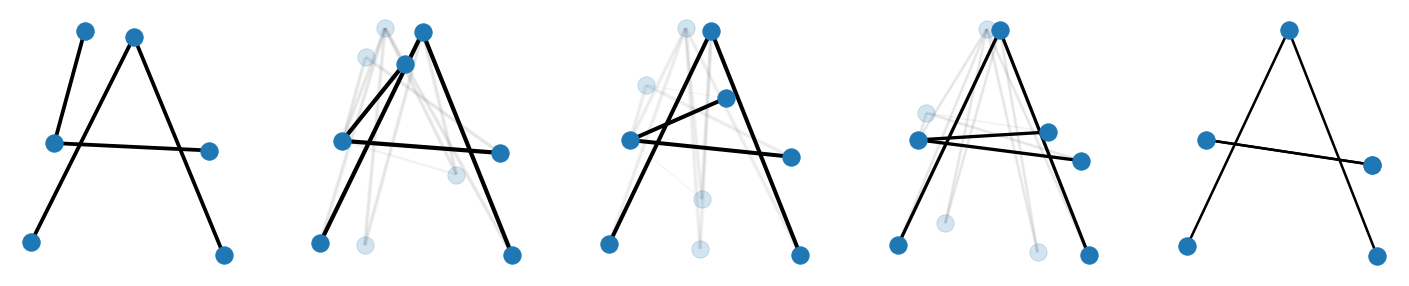}
    \includegraphics[width = 0.45\textwidth]{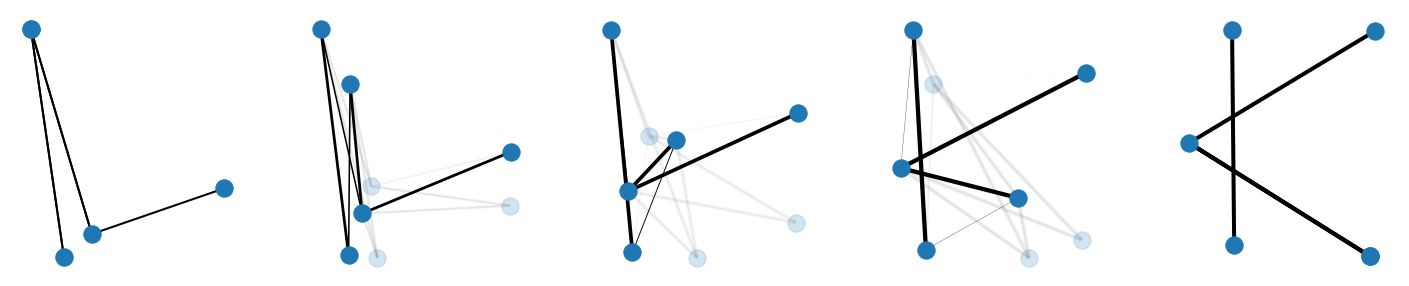}
    \caption{Examples of geodesics between simple graphs. See Section \ref{sec:geodesic_examples} for explanations.}
    \label{fig:geodesic_examples}
\end{figure}

Figure \ref{fig:geodesic_examples} shows several examples of geodesics between simple graphs. Each graph $X$ is a measure network with $\omega_X$ the graph adjacency matrix. Except for the example in the first row, each $\mu_X$ is a uniform  node measure.

The first row in the figure shows a geodesic between graphs with the same edge weight structure as the graphs in Figure \ref{fig:gromov_wasserstein_geodesic}, but with different node weights indicated by node sizes. Observe that the difference in node weights changes the geodesic path drastically. The second row in the figure shows a geodesic between a graph with two disconnected nodes, each with a self-loop (not shown) and a connected graph with large clusters. The geodesics in the third and fourth  rows of the figure are each between graphs from the ``Letter Graphs" graph classification benchmark dataset \cite{riesen2008iam}; the first between letters in the same class and the second between letters in different classes. Each geodesic is displayed with lower opacity on some nodes and edges---these nodes with weight less than a user-defined threshold (50\% (respectively, 40\%) of maximum node weight in the first (respectively, second) example) and edges with at least one endpoint meeting this criteria. This  technique allows us to understand  common graph features at multiple resolutions.

\subsection{Shape Classification}

As a proof-of-concept for incorporating this framework into machine learning pipelines, we present a simple shape classification experiment. The data consists of 20 object classes with 20 samples from each class from the well known MPEG-7 computer vision database  (see Figure \ref{fig:shape_samples}). Each shape consists of 100 planar points. The input data for the experiment  consists of pairwise distance matrices for each shape, yielding $400$ matrices of size $100 \times 100$. The ordering of the points was randomized when constructing the distance matrices. Weights on the nodes are uniform.

\begin{figure}
    \centering
    \includegraphics[width=\linewidth]{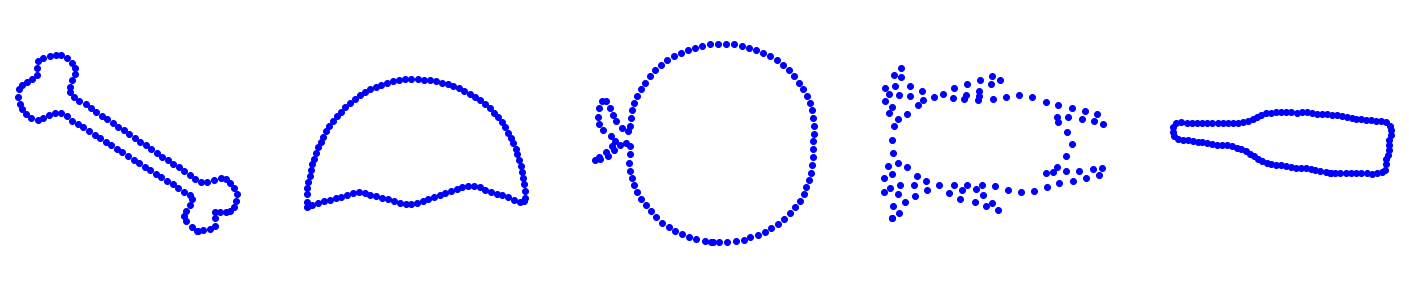}
    \caption{Samples from five shape classes in the classification experiment.}
    \label{fig:shape_samples}
\end{figure}

We consider three methodologies for classifying the shapes. In each experiment the same 80\% of the shapes were used as a training set. In a completely naive approach, a support vector machine was trained on the permuted distance matrices. With this method, the classification rate on the testing set was 19\%. In the second approach, one of the permuted distance training matrices $X$ was fixed and all other training matrices were aligned to $X$. The matrices were then ``centered" on $X$, which can be understood as pulling them back to tangent vectors in $T_{[X]}$ via the log map (or, rather, the coupling-dependent log map $\log_{[X]}^\mu([Y])$ here). An SVM was then trained on these tangent vectors. Classification was tested by aligning and centering the test matrices with $X$, yielding a classification rate of 84\%. In the final method, the Fr\'{e}chet mean $\overline{X}$ of all samples in the training set was computed. Then all training matrices were aligned to $\overline{X}$ and pulled back to tangent vectors, where an SVM was trained. Classification was once again performed by aligning and centering test matrices with $\overline{X}$, where the classification rate was improved to 94\%. This approach illustrates a template for vectorization of network data.

\subsection{Tangent PCA on planar shapes}

The Riemannian framework allows us to do other machine learning computations by pulling networks back to a tangent space; for example, we now present results of a tangent PCA experiment. Figure \ref{fig:exp-apples} illustrates the \emph{apple dataset} that we used. Each shape is represented as a measure network as in the previous subsection. To perform tangent PCA, we first computed a \frechet mean for these 20 shapes. Next we used log maps based at the chosen \frechet mean to pull back the 20 shapes to vectors in the tangent space. Here we performed PCA as usual. 

The first three principal directions explained 85\% of the variance in the data, and they are visualized in Figure \ref{fig:exp-apples-pca} via MDS embeddings. The first direction captures variance in the size of the apple; the second captures surface irregularities and the size and shape of the leaves, and the third captures the presence of a ``bite" on the apple.

\begin{figure}
    \centering
    \includegraphics[width=\linewidth]{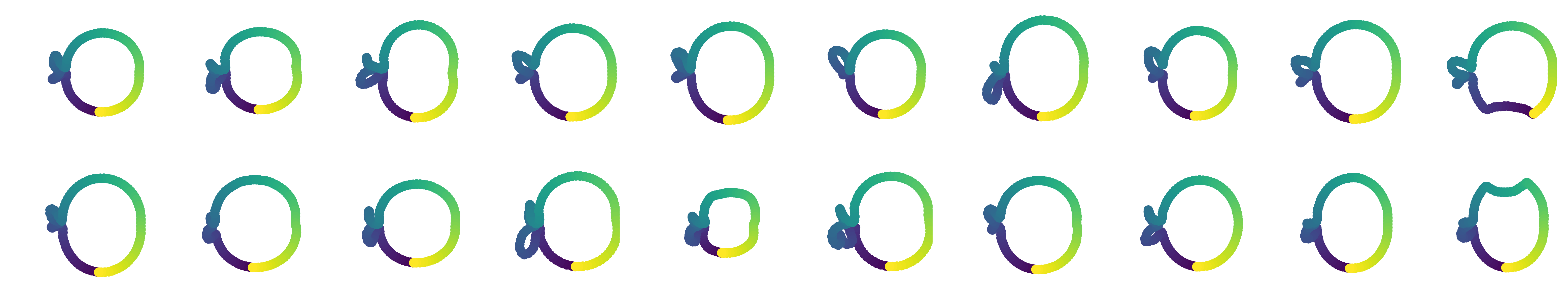}
    \caption{Apple dataset. Each shape is a  measure network $X$ containing 100 points, with $\mu_X$ uniform measure and $\omega_X$ the pairwise (Eulcidean) distance matrix between the points.}
    \label{fig:exp-apples}
\end{figure}

\begin{figure}
    \centering
    \includegraphics[width=\linewidth]{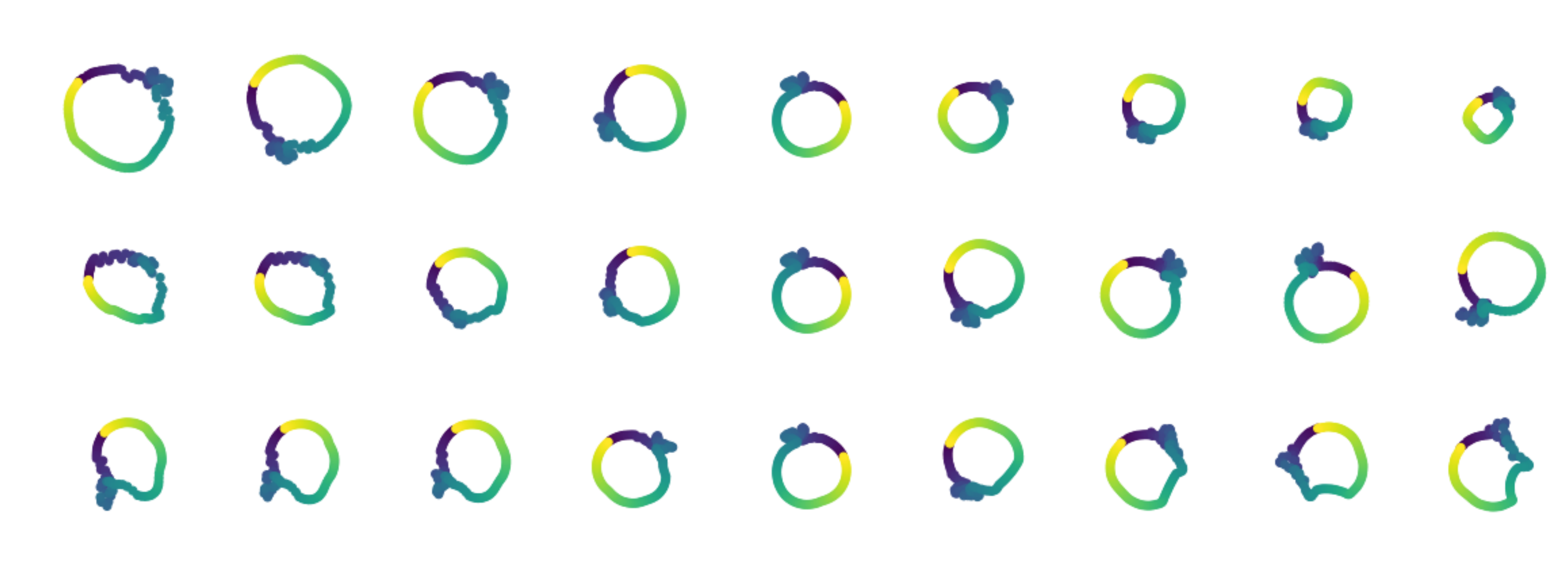}
    \caption{First three principal directions of variance for tangent PCA on the apple dataset.}
    \label{fig:exp-apples-pca}
\end{figure}

\subsection{Compressing an Asymmetric SBM Network}
\label{sec:exp-sbm-compr}

To illustrate our constructions on asymmetric networks, we generated a $100\times 100$ asymmetric stochastic block model (SBM) network $Y$ following the model provided in \cite{gwnets}. Here $Y$ consisted of five blocks $B_1,\ldots, B_5$ of 20 nodes each. For $y \in B_i$ and $y' \in B_j$, we sampled $\w_Y(y,y') \sim N(\mu_{ij},5)$, where $\mu_{ij} \in \{0, 25, 50, 75, 100\}$. Negative values were allowed. See Figure \ref{fig:exp-sbm} for an illustration of $Y$. Note that $Y$ is intuitively represented by a $5 \times 5$ ``ground-truth" matrix. 

We averaged $Y$ with a $5\times 5$ all-zeros matrix $X$ using the network compression approach given in the supplementary materials to see if our method would recover the ground truth matrix. The output of our method is shown in the middle panel of Figure \ref{fig:exp-sbm}---up to a permutation (shown in the right panel), this accurately recovered the matrix of $\mu_{ij}$ values. 

\begin{figure}
    \centering
    \includegraphics[width = 0.3\linewidth]{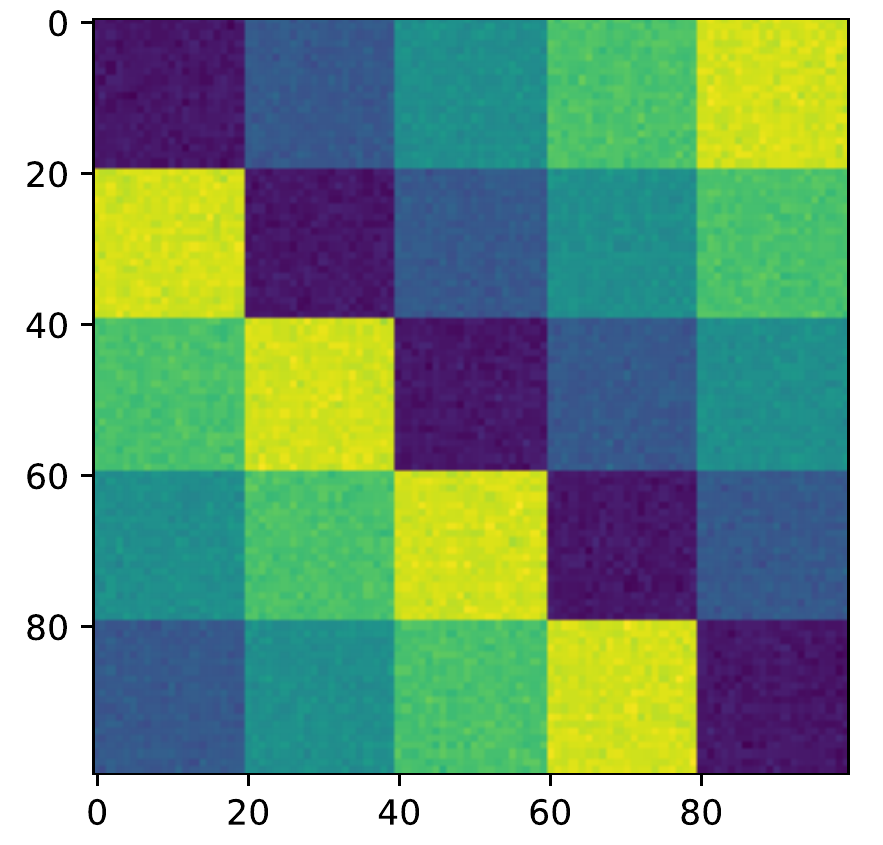}
    \includegraphics[width = 0.3\linewidth]{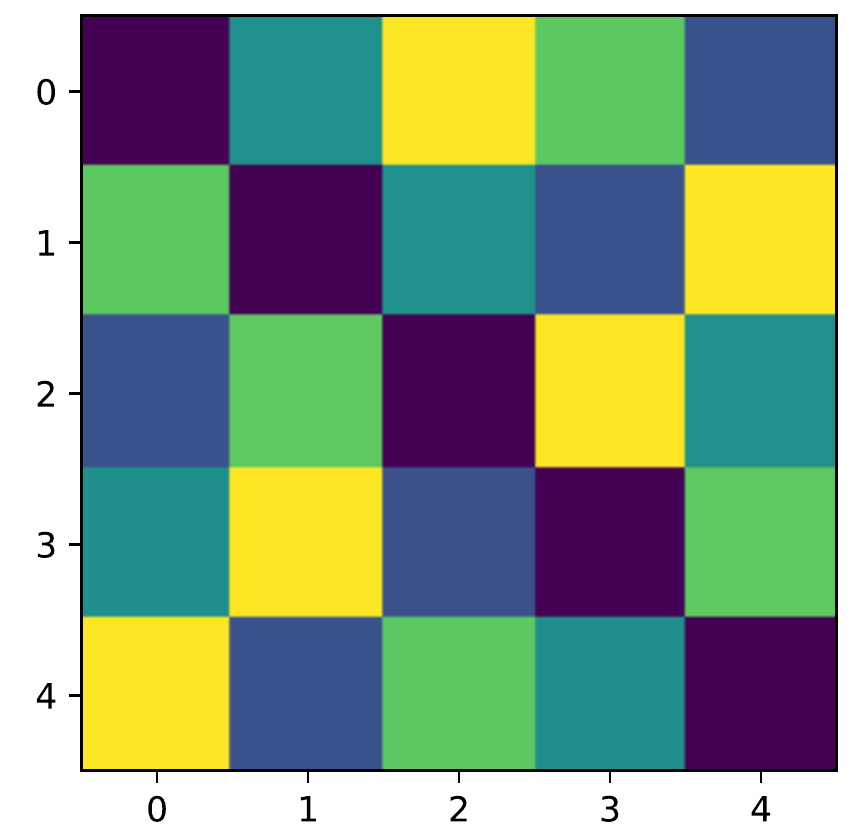}
    \includegraphics[width = 0.3\linewidth]{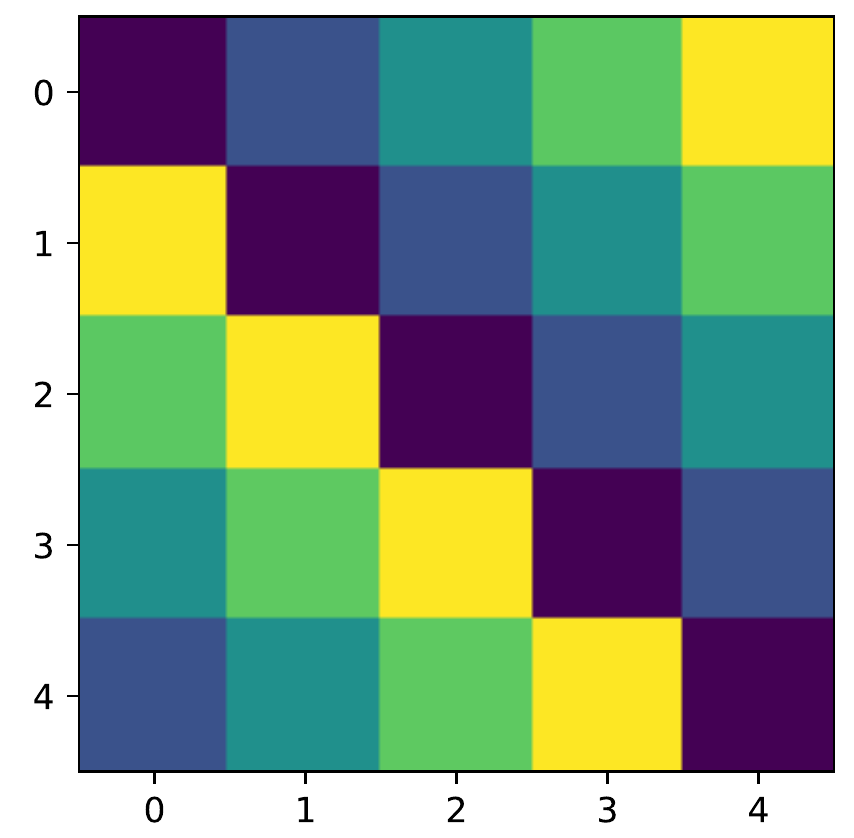}
    \caption{\textbf{Left:} $100 \times 100$ SBM $Y$ with entries drawn from $N(\mu, \s^2)$, where $\s^2:= 5$. The five colors correspond to $\mu = 0\mathrm{ (blue)}, 25, 50, 75, 100 \mathrm{ (yellow)}.$ \textbf{Middle: } $5 \times 5$ compressed average of $Y$ and the $5 \times 5$ all-zeros matrix $X$. Both $\mu_X,\mu_Y$ were taken to be uniform. Colors range in $\{0, 12.5, 25, 37.5, 50\}\pm 0.1$. \textbf{Right:} Permuted form of middle figure.}
    \label{fig:exp-sbm}
\end{figure}

\section{Discussion}

In this paper, we followed the seminal work of Sturm \cite{sturm2012space} on Riemannian structures induced by GW distances and produced a Riemannian framework for performing data analysis on collections of arbitrary matrices. There are many applications in data science which can be reframed using this formalism, such as network clustering and sketching and future work will focus on making these formulations precise.

There are several open challenges left to be explored from both theoretical and practical perspectives. On the theoretical side, one would like to obtain estimates on injectivity radii for measure networks with special properties (this amounts to replacing the $L^\infty$ bounds in Proposition \ref{prop:tangent_matrices} with $L^2$ bounds). It would also be interesting to determine conditions where the Fr\'{e}chet gradient flow is guaranteed to converge---in our applications, this either happened naturally or was enforced by a compression step. On the practical side, one would like to improve the scalability of our algorithms by incorporating entropic regularization \cite{cuturi2013sinkhorn} and the more sophisticated GW algorithm of \cite{xu2019scalable}. Several steps in our framework relied on the empirical observation of sparsity in optimal couplings, so incorporating entropic regularization will bring its own collection of theoretical challenges.

\section*{Acknowledgments}

We would like to thank Facundo M\'{e}moli for many useful discussions about Gromov-Wasserstein distance and for introducing us to the work of Sturm on the topic. We would also like to thank Justin Solomon for a helpful conversation about the matrix formulation of the GW problem for asymmetric networks. Finally, we thank the anonymous referees for their thoughtful comments.

\pagebreak

{\small
\bibliographystyle{ieee_fullname}
\bibliography{biblio}
}

\pagebreak

\appendix

\section{Derivation of Equation \ref{eq:gw-matrix}}
To avoid overloading notation, write $p:= \mu_X,\, q:=\mu_Y$. In matrix form, we have:
\[ \dn(X,Y) = \frac{1}{2} \min_{C\in \coup(p,q)} \Big( \sum_{ijkl} |X_{ik} - Y_{jl} |^2 C_{kl}C_{ij} \Big)^{\frac{1}{2}}\]
Expanding the term inside the square root yields three terms. The first is the following:
\begin{align*}
    \sum_{ijkl}X_{ik}^2 C_{kl}C_{ij} &= \sum_{ik}X_{ik}^2p_kp_i
    = \sum_i p_i \sum_k X^2_{ik}p_k\\
    &=\sum_i p_i  (X.^{\wedge 2}p)_i
    = \la p, X.^{\wedge 2}p \ra.
\end{align*}
Here the first equality followed by marginalization. Another term is as follows:
\begin{align*}
    \sum_{ijkl}Y_{jl}^2 C_{kl}C_{ij} &= \sum_{jl}Y_{jl}^2q_lq_j
    = \sum_l q_l \sum_j Y^2_{jl}q_j\\
    &=\sum_l q_l  (Y.^{\wedge 2}q)_l
    = \la q, Y.^{\wedge 2}q \ra.
\end{align*}
The final term is the only one that depends on $C$:
\begin{align*}
    -2  \sum_{ijkl}X_{ik}Y_{jl} & C_{kl}C_{ij}
    = -2  \sum_{il} (XC)_{il}(CY)_{il}\\
    &= -2 \la XC, CY \ra = -2 \tr(C^TX^T C Y ).
\end{align*}
The final equality holds by the definition of the Frobenius product, and this concludes the derivation of Equation \ref{eq:gw-matrix}. 

We further note that in the special case where $X, Y$ are symmetric positive definite, we can take a Cholesky decomposition to write:
\[X = UU^T, \, Y = V^TV.\]
Then we have:
\begin{align*}
    \tr(C^TX^T C Y ) &= \tr(C^T U^T U C V^TV ) \\
    &= \tr(V C^T U^T U C V^T) \\
    &= \la UCV^T, UCV^T \ra = \| UCV^T \|^2,
\end{align*}
where $\| \cdot \|$ denotes the Frobenius norm. The function $C \mapsto \| UCV^T \|^2$ is now seen to be convex.

\section{Proofs}

\begin{proof}[Proof of Proposition \ref{prop:gradient_of_GW}.]
Let $A = A_{XY}$. Writing $A_s = \frac{1}{2}\lp A + A^\ast\rp$ for the \emph{symmetrization} of $A$, we observe that 
\begin{align*}
2\la A_s C,C \ra &= \la AC + A^\ast C, C \ra = \la  AC, C \ra + \la C,  AC \ra \\
&= \la  AC, C \ra + \la  AC, C \ra  = 2 \la AC, C\ra.
\end{align*}
The computation then agrees with the computation in the symmetric setting of \cite{pcs16} after replacing $A$ with its symmetrization.
\end{proof}

\begin{lemma} 
\label{lem:geod-start-wisom}
Let $(Z,\w_Z,\mu_Z)$ be a finite measure network. Let $f \in L^2(Z^2, \mu_Z^{\otimes 2})$. For $t \in [0,1]$, define $\w_t : (Z\times Z)^2  \rightarrow \R$ as
\begin{align*}
 \w_t\lp (z_1,z_2), (z_3,z_4) \rp &= (1-t)\w_Z(z_1,z_3) \\
 &+ t\w_Z(z_2,z_4) + tf(z_2,z_4).
 \end{align*}
 Also let $\Delta$ denote the diagonal coupling between $\mu_Z$ and itself, i.e. the pushforward of $\mu_Z$ under the diagonal map $z \mapsto (z,z)$. Then we have:
 \[ (Z\times Z, \w_t, \Delta) \cong^w (Z,\w_Z,\mu_Z).\] 
\end{lemma}

\begin{proof}
Consider the projection map $\pi: Z\times Z \rightarrow Z$ defined by $(z_1,z_2) \mapsto z_1$. It suffices to show that $\pi_\#\Delta = \mu_Z$ and $\lnorm \pi^*(\w_Z +tf) - \w_t \rnorm_\infty = 0$. For the first assertion, let $A \in \borel(Z)$. Then we have:
\begin{align*}
\pi_\#\Delta(A) = \Delta(A \times Z) = \mu_Z(A).
\end{align*}
For the second assertion, let $\lp(z_1,z_2),(z_3,z_4)\rp \in (Z\times Z)^2$. Suppose also $z_1 = z_2, \, z_3 = z_4$. Then we have:
\begin{align*}
&\pi^*(\w_Z + tf)\lp(z_1,z_2),(z_3,z_4)\rp \\
&= \w_Z(z_1,z_3) + tf(z_1,z_3)\\
&= \w_t((z_1,z_1), (z_3,z_3))\\
&= \w_t((z_1,z_2), (z_3,z_4)).
\end{align*}
The conclusion follows because $\Delta$ assigns zero measure to all pairs $(z,z')$ where $z\neq z'$.
\end{proof}

\begin{proof}[Proof of Proposition \ref{prop:tangent_matrices}]
Let $X=(X,\omega_X,\mu_X)$ be a finite measure network and let $f \in L^2(Z^2,\mu_Z^{\otimes 2})$ for some $Z \in [X]$. We wish to derive a condition which guarantees that
$$
\gamma(t) := [Z,\omega_Z +  f, \mu_Z]
$$
is a geodesic defined on $[0,1]$. For any $t$, $(Z,\omega_Z +  t f, \mu_Z)$ lies in the same weak isomorphism class as
$$
\lp Z \times Z, (1-t ) \omega_Z + t (\omega_Z +  f), \Delta \rp,
$$
where $\Delta$ denotes the diagonal coupling of $Z$ with itself, as in Lemma \ref{lem:geod-start-wisom}. This is the general form of a geodesic given above \eqref{eqn:geodesic}. Moreover, $\gamma(0) = [X]$, by the definition of $Z$. It therefore suffices to find a condition on $f$ which guarantees that $\Delta$ is an optimal coupling between $Z$ and the measure network 
\begin{equation}\label{eqn:Z_1_definition}
Z_1:=(Z, \omega_Z +  f, \mu_Z)
\end{equation}

Consider an arbitrary coupling $\mu$ of $Z$ with $Z_1$. The squared distortion $\dis(\mu)^2$ is given by
\begin{align*}
&\int_{(Z \times Z)^2} \lp \omega_Z(z_1,z_2) +   f(z_1,z_2) - \omega_Z(z_3,z_4) \rp^2 \mu \otimes \mu,
\end{align*}
where $\mu \otimes \mu$ is short for $\mu \otimes \mu ((dz_1,dz_2),(dz_3,dz_4))$. We rewrite this as
\begin{align}
&\int_{(Z \times Z)^2} \left\{ \lp \omega_Z(z_1,z_2) - \omega_Z(z_3,z_4) \rp^2 \right. \label{eqn:bracket_term} \\
&\hspace{.2in} + \Bigg.  2  \lp \omega_Z(z_1,z_2) - \omega_Z(z_3,z_4) \rp f(z_1,z_2) \Big\} \mu \otimes \mu \nonumber \\
&\hspace{.2in} +  \int_{(Z \times Z)^2} f(z_1,z_2)^2 \mu \otimes \mu. \label{eqn:simplifying_term}
\end{align}
By the fact that $\mu$ is a coupling of $\mu_Z$ with itself, the term in line $\eqref{eqn:simplifying_term}$ simplifies to 
$$
 \int_{Z^2} f(z_1,z_2)^2 \mu_Z(z_1) \mu_Z(z_2).
$$
On the other hand, this quantity is equal to the squared distortion $\dis(\Delta)^2$.

To guarantee that $\dis(\Delta) \leq \dis(\mu)$, it suffices that the bracketed term in \eqref{eqn:bracket_term} can be made non-negative. If each $\left| \omega_Z(z_1,z_2) - \omega_Z(z_3,z_4)\right|$ is zero, then $\omega_X$ is constant, in which case we immediately see that the bracketed term is nonnegative without restriction on $f$. Otherwise, let $\epsilon_{[X]}$ be one half of the infimal strictly positive value of $\left| \omega_Z(z_1,z_2) - \omega_Z(z_3,z_4)\right|$, ranging over all quadruples of points in $Z$. Since $Z$ is weakly isomorphic to $X$, the images of $\omega_X$ and $\omega_Z$ are equal, and since $X$ is finite these images are finite. It follows that the infimum $\epsilon_{[X]}$ is actually a minimum and is strictly positive. Under the assumption that $|f(z,z')| < \epsilon_{[X]}$ for each $z,z' \in Z$, it is straightforward to check that the bracketed term in \eqref{eqn:bracket_term} is nonnegative, and this completes the proof.
\end{proof}

\begin{proof}[Proof of Proposition \ref{prop:gradient_frechet_functional}]
For simplicity, suppose that $S = \{Y\}$ contains a single finite network and write $F=F_S$. The general case follows by similar arguments. After alignment, we can assume that $X = (X,\omega_X,\mu_X)$, $Y = (X,\omega_Y,\mu_X)$ and that the diagonal coupling $\Delta$ is optimal. 

Let $[f] \in T_{[X]}$. Once again, we assume for simplicity that $f$ is defined on a finite measure network, which we may as well take to be $X$ after realigning as necessary. The general case can be shown by adapting this specialized argument. 

The first task is to compute the directional derivative $D_{[f]}F([X])$. For $t \geq 0$, let $X_t = (X,\omega_X + t f, \mu_X)$ and let $\mu_t$ denote an optimal coupling of $X_t$ with $Y$ such that that $\lim_{t \rightarrow 0^+} \mu^t$ is the diagonal coupling $\mu_X \otimes \mu_X$. Note that for each $t$, the quantity
\begin{equation}\label{eqn:directional_derivative_frechet}
\frac{1}{t} \lp F(\exp_{[X]} (t[f])) - F([X]) \rp
\end{equation}
is upper bounded by
$$
\frac{1}{t} \lp \dis (\mu_t)^2 - \dis(\mu_X \otimes \mu_X)^2 \rp.
$$
It is a straightforward computation to show that this upper bound can be rewritten as 
\begin{align}
&t  \sum_{i,j} f(i,j)^2 \mu_X(i) \mu_X(j) \label{eqn:directional_derivative_frechet_2} \\
&\hspace{.2in} + 2  \sum_{i,j} (\omega_X(i,j) - \omega_Y(i,j) ) f(i,j) \mu_X(i) \mu_X(j). \nonumber
\end{align}
On the other hand, \eqref{eqn:directional_derivative_frechet} is lower bounded by
$$
\frac{1}{t} \lp \dis (\mu_t)^2 - \dis_{X,Y}(\mu_t)^2 \rp,
$$
where $\dis_{X,Y}(\mu_t)$ is the distortion of $\mu_t$ treated as a coupling of $X$ and $Y$. This simplifies to
\begin{align}
&t \sum_{i,j,k,\ell} f(i,j)^2 \mu_t(i,k) \mu_t(j,\ell) \label{eqn:directional_derivative_frechet_3} \\
&\hspace{.1in} + 2 \sum_{i,j,k,\ell} (\omega_X(i,j) - \omega_Y(k,\ell)) f(i,j) \mu_t(i,k) \mu_t(j,\ell). \nonumber
\end{align}
As $t \rightarrow 0^+$, quantities \eqref{eqn:directional_derivative_frechet_2} and \eqref{eqn:directional_derivative_frechet_3} both limit to
$$
2  \sum_{i,j} (\omega_X(i,j) - \omega_Y(i,j) ) f(i,j) \mu_X(i) \mu_X(j),
$$
and this therefore provides a formula for the directional derivative $D_{[f]}F([X])$. 

Finally, we note that
\begin{align*}
&2  \sum_{i,j} (\omega_X(i,j) - \omega_Y(i,j) ) f(i,j) \mu_X(i) \mu_X(j) \\
&\hspace{.2in} = \left<[f],\nabla F([X])\right>_{[X]}
\end{align*}
if we take $\nabla F([X])$ to be represented by the matrix
$$
\lp \nabla F (X) \rp_{ij} = 2 \lp \omega_X(x_i,x_j) - \omega_Y(y_i,y_j) \rp,
$$
which is the claimed form for this specific example. The general formula (for $S$ of larger cardinality) is derived by linearity.
\end{proof}

\section{Support sizes for optimal couplings} \label{sec:support_sizes}

The benefit of our representation of geodesics between measure networks is the empirical observation that (approximations of) optimal couplings tend to be sparse. This allows a geodesic between measure networks $X$ and $Y$ to be represented in a much smaller space than the naive requirement of size $|X| \cdot |Y|$. We have observed that it is more typical for the representation to require size which is linear in $|X| + |Y|$. Experimental evidence for this observation is provided in Figures \ref{fig:random_support_size} and \ref{fig:IMDB_support_size}.

There is also theoretical evidence for the observed small support size phenomenon. In \cite{chen2013sparse,chen2015new} the authors show that random quadratic programming problems tend to have sparse solutions with high probability. The setting of these articles is not exactly the one considered here (they use symmetric quadratic forms and optimize over the standard simplex) and it remains an open problem to give theoretical probabilistic guarantees for sparsity in the GW setting. Moreover, it would be interesting to get results for cost matrices with more realistic structures; e.g. binary matrices representing random directed adjacency matrices. 

\begin{figure}
    \centering
    \includegraphics[width = 0.45\textwidth]{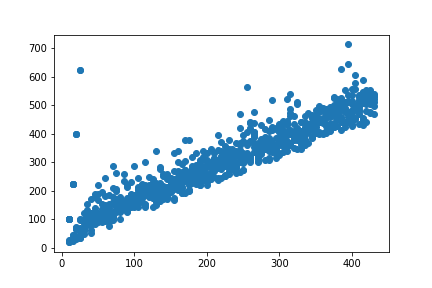}
    \caption{Support sizes for random measure networks. In each trial, a pair of Gaussian iid random weight matrices of size $n$ is drawn. The optimal coupling for the uniformly weighted networks is computed and its support size is plotted against $n$. In general, the support size grows linearly.}
    \label{fig:random_support_size}
\end{figure}

\begin{figure}
    \centering
    \includegraphics[width = 0.45 \textwidth]{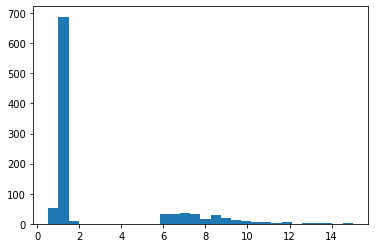}
    \caption{Support sizes for real networks. In each of 1000 trials, a random pair of graphs from the IMDB-BINARY graph classification benchmark dataset is chosen. The optimal coupling between their shortest path distance matrices (with uniform weights on the nodes) is computed. This histogram shows the distribution of support size divided by the sum of sizes of the graphs being compared. In general, the support size is a small multiple of the sum of graph sizes.}
    \label{fig:IMDB_support_size}
\end{figure}

\section{Support sizes for the iterative averaging scheme}
\label{sec:comp-asym}

\begin{figure*}
    \centering
    \includegraphics[width=0.48\linewidth]{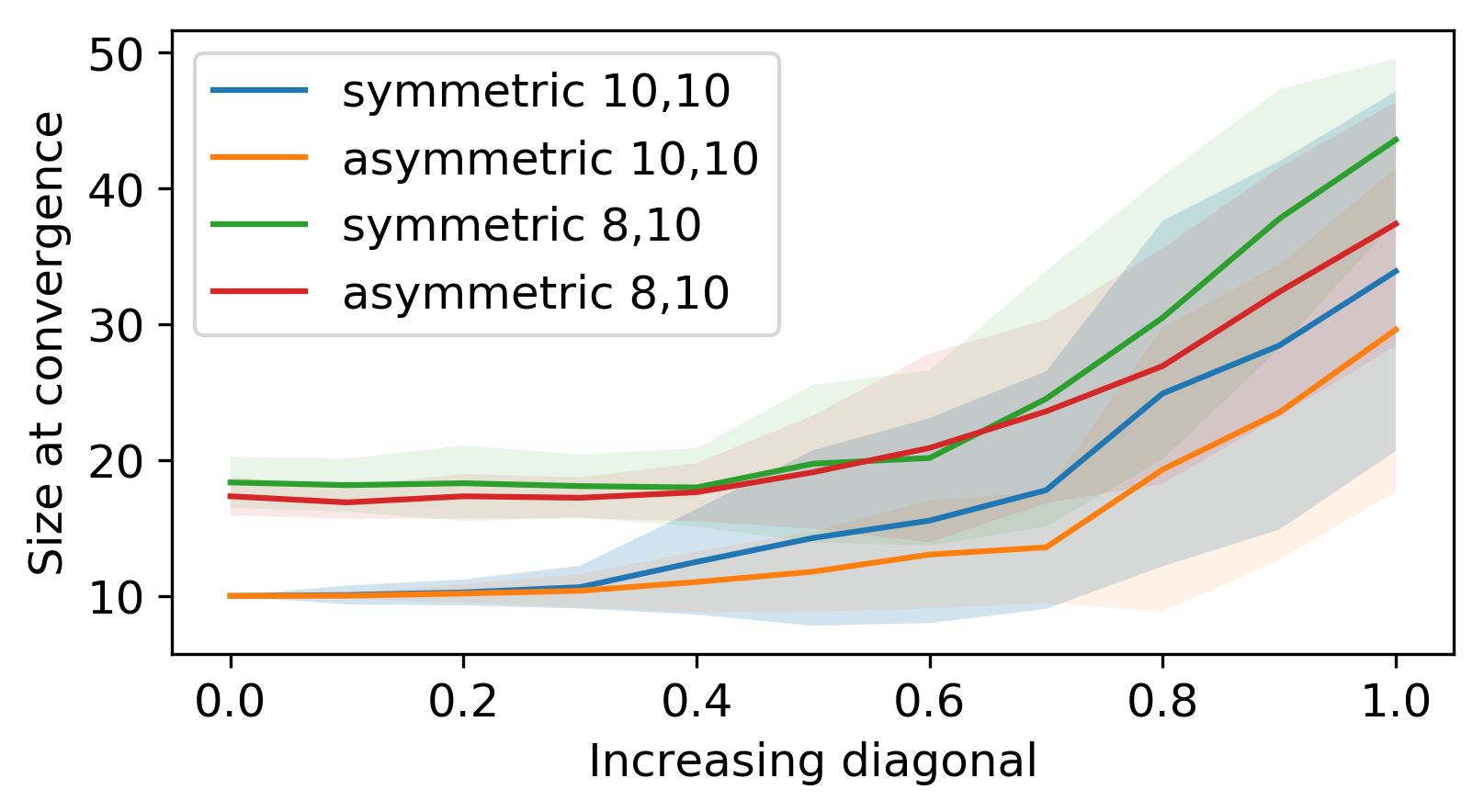}
    \includegraphics[width=0.48\linewidth]{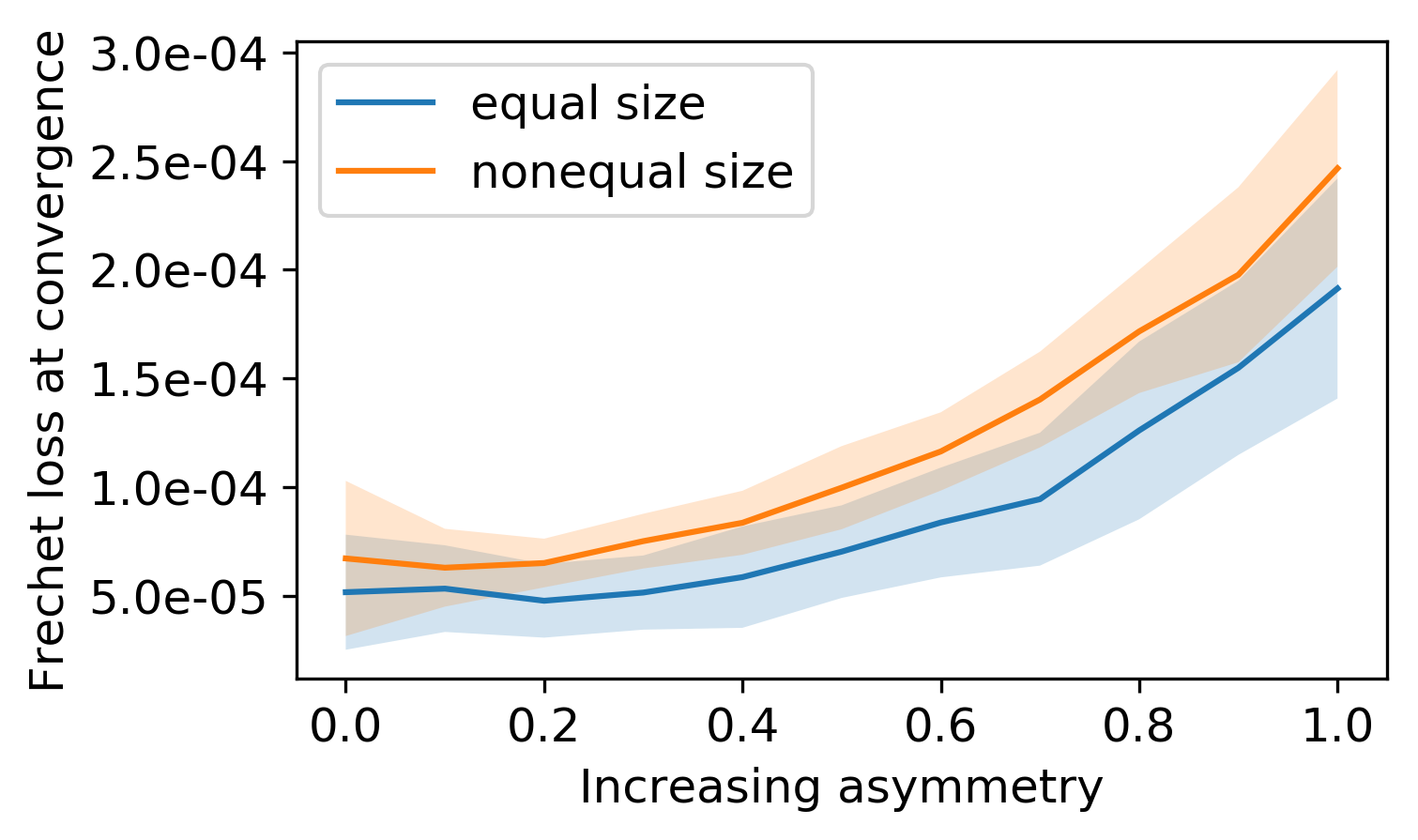}
    \caption{\textbf{Left:} The sizes of the iterates for the \frechet mean procedure depend on the diagonal entries of the network weight matrices. However, these sizes are not influenced by the level of asymmetry in the matrices. \textbf{Right:} The values of the \frechet loss function at convergence rise with increasing asymmetry of the network weight matrices.}
    \label{fig:mean-diag-asym}
\end{figure*}

Practical computation of \frechet means as described in the main text comes with the standard challenges of nonconvex optimization: the gradient descent for finding optimal couplings may get stuck in bad local minima, and this in turn may propagate into poor computation of \frechet means. Empirically we found that using a schedule for adjusting the gradient step size, i.e. using full gradient steps at the beginning and then using backtracking line search with Armijo conditions \cite{nocedal2006numerical} often worked well. Accelerating the gradient descent using the momentum method also works well. 

One aspect of the convergence problem is the size of the blowups needed to take discrete steps along the gradient flow of the \frechet functional. Towards characterizing the classes of networks for which this problem is more or less difficult, we set up the following experiments. First we generated networks $X_1,X_2$ with random weight matrices generated using Python's \texttt{numpy.random.rand} function. We equipped these networks with uniform probabilities. Next we wrote $X_j = Y_j + D_j$ for $j=1,2$, where $D_j$ consisted of the diagonal part of $X_j$, and $Y_j$ had zero diagonal. Next we wrote $X_j^{(\alpha)} := Y_j + \alpha D_j$ for $\alpha \in \{0, 0.1, 0.2,\ldots, 1\}$. For each $\alpha$, we set $\mathcal{X}^{(\alpha)} := \{X_1^{(\alpha)}, X_2^{(\alpha)}\}$ and computed the \frechet mean of each $\mathcal{X}^{(\alpha)}$ using 100 randomly generated initial seed networks. We repeated this procedure in the cases where the $X_j$ were both 10-node networks and where $X_1$ had 8 nodes, and $X_2$ had 10 nodes. Finally, we repeated this entire procedure after initially symmetrizing the $X_j$. The average sizes of the iterates are plotted against $\alpha$ in the left panel of Figure \ref{fig:mean-diag-asym}. The shading represents the standard deviation for each curve. First we note that as the diagonal terms are gradually added in, the sizes of the \frechet mean iterates grow rapidly. This suggests that when preprocessing data for the \frechet averaging procedure, it is helpful to use a scheme which enforces zero diagonals. The second observation is that there is some extra blowup that happens when averaging over a list of networks with different sizes. This is expected, as the optimal couplings between such networks cannot be permutation matrices, and hence some blowup is necessary.  

Another interesting observation is that the level of asymmetry does not seem to affect the sizes of the iterates. However, asymmetry does affect the final \frechet loss value at convergence. To test this effect, we generated matrices $X_j$ as above and decomposed them into symmetric and antisymmetric parts: $X_j = S_j+ A_j$. Next we chose $\alpha$ as above and considered the networks $Z_j^{(\alpha)}:= S_j + \alpha A_j$. For each $\alpha$, we set $\mathcal{Z}^{(\alpha)}:= \{Z_1^{(\alpha)}, Z_2^{(\alpha)}\}$ and computed the \frechet mean of each $\mathcal{Z}^{(\alpha)}$ using 100 randomly generated initial seed networks. We repeated this experiment for the cases where both $X_j$ had 10 nodes, and where $X_1$ had 8 nodes and $X_2$ had 10 nodes. The values of the final \frechet loss are plotted against $\alpha$ in the right panel of Figure \ref{fig:mean-diag-asym}. We observe that the final \frechet loss increases with asymmetry, which suggests that the \frechet function becomes more nonconvex with increasing asymmetry.

These observations point to the following open questions:
\begin{itemize}
    \item Can one place quantitative bounds on the rate of expansion of the \frechet mean iterates as a function of the diagonal values of weight matrices?
    \item Can one adapt methods such as graduated nonconvexity to improve convergence for asymmetric networks, in the sense of ``graduated asymmetry"?
\end{itemize}

To perform averages for networks with nonzero diagonal while circumventing the problem of expanding matrices, we adopted a simple---albeit Procrustean---method for restricting this expansion. This method has its own interesting application for \emph{network compression}, and we detail it next.

\subsection{Network compression}
\label{sec:compression}

Let $X, \, Y$ be finite networks, and let $\hat{X}$, $\hat{Y}$ denote their alignments. The aligned networks could, a priori, be larger in size than $X$ and $Y$. Thus if the alignment is iterated, as would be the case in computing \frechet means, we could have unbounded blowups in the sizes of these matrices. To prevent this situation, we pose the following question. Suppose $|X| < |Y|$. \emph{What is the projection of the vector $\w_{\hat{Y}} - \w_{\hat{X}}$ onto the space of $|X| \times |X|$ vectors?} Let $v$ denote this projection. Geometrically, we expect that $(X, \w_X + v, \mu_X)$ is a good $|X|$-node representative of $Y$. Practically, we can take the average of $(X,\w_X,\mu_X)$ and $(X,  \w_X + v, \mu_X)$ without any expansion and expect this object to be an approximate average of $X$ and $Y$. 

We adopt the following simple method to obtain a low-dimensional representation of the tangent vector $\nu:= \w_{\hat{Y}} - \w_{\hat{X}}.$ Following the notation used in Definition \ref{def:blowup}, write $\hat{X} = X[\mathbf{u}]$. Recall that $\w_{X[\mathbf{u}]}((x,i),(x',j)) = \w_X(x,x').$ Define the $|X|\times |X|$-dimensional vector $v$ as follows: for any $x,x' \in X$,
\begin{align*}
    v(x,x') &:= \frac{\sum_{i = 1}^{u_x} \sum_{j = 1}^{u_{x'}} (\w_{\hat{Y}} -\w_{\hat{X}})((x,i),(x',j)) }{u_x \cdot u_{x'}}\\
    &= \frac{\sum_{i = 1}^{u_x} \sum_{j = 1}^{u_{x'}} \w_{\hat{Y}}((x,i),(x',j)) }{u_x \cdot u_{x'}} - \w_X(x,x').
\end{align*}
Here we overload notation slightly to write $\w_{\hat{Y}}((x,i),(x',j))$, but this is well-defined because $\hat{Y}$ is aligned to $\hat{X}$ and $(x,i)$ is just an index.

To understand this construction, note that the elements of the tangent vector $\nu$ admit the following interpretation: $\nu_{pq}$ is just the difference $-\w_{\hat{X}}(x_p,x_q) + \w_{\hat{Y}}(y_p,y_q)$, i.e. it measures the change in the network weight from $x_p$ to $x_q$ when transferring from $\w_{\hat{X}}$ to $\w_{\hat{Y}}$. Here $x_p,x_q$ are just indices of elements in $\hat{X}$. In the metric space setting, this quantity is exactly the change in distance between $x_p$ and $x_q$ that one would observe by following the optimal transport map $\hat{\mu}$ between $\hat{X}$ and $\hat{Y}$. Intuitively in the metric setting, points which start nearby and end nearby under the map $\hat{\mu}$ correspond to similar tangent vector entries.

Under this interpretation, the vector $v$ simply averages out the changes that occur within and between blocks of $X[\mathbf{u}]$ when passing from $\w_{\hat{X}}$ to $\w_{\hat{Y}}$. Note in particular that $(X, \w_X + v, \mu_X)$ gives us a \emph{compressed representation} of $Y$. This is illustrated in Section \ref{sec:exp-sbm-compr}.

\begin{remark}
The averaging method of \cite{pcs16} proceeds by fixing a size for the requested \frechet mean and then performing an alternating optimization. This suggests the following open question: Is there a variant of the ``compressed log map" approach outlined above that agrees with the method in \cite{pcs16}?
\end{remark}

\section{Algorithms}

We now present pseudocode for our methods. Algorithm \ref{alg:optcoup} serves as a placeholder; it can be computed using gradient descent \cite{pcs16} and is implemented in the Python Optimal Transport Library \cite{flamary2017pot}.

\begin{algorithm}[H]
\caption{Compute minimizer of the GW functional}\label{alg:optcoup}
\begin{algorithmic}[1]
\Function{optCoup}{$A,B,a,b$}
\State \emph{// $A \in \R^{n\times n}$, $B \in \R^{m\times m}$. $a,b$ probability vectors} 
\Return $C$ \Comment{$n \times m$ optimal coupling}
\EndFunction
\end{algorithmic}
\end{algorithm}

\begin{algorithm}[H]
\caption{Computing the log map}\label{alg:logmap}
\begin{algorithmic}[1]
\Function{logMap}{$A,B,a,b$}
\State \emph{// $A \in \R^{n\times n}$, $B \in \R^{m\times m}$. $a,b$ probability vectors} 
\State \parbox{0.8\linewidth}{\emph{// Lift geodesic from $A$ to $B$ to tangent vector based at $A$}}\\

\State Initialize $splitData$ = [] \Comment{store metadata}
\State $C$ = \Call{optCoup}{$A,B,a,b$}
\State Find rows, columns of $C$ with multiple nonzeroes
\State Store indices in $splitData$
\State Blow-up $A,B,a,b, C$ according to $splitData$ 
\State $C = (C!=0)$ \Comment{convert $C$ to permutation matrix}
\State $B = C*B*C^T$ \Comment{align $B$ to $A$} 
\State $v = - A + B$ \Comment{tangent vector}
\State \Return $A,a, v, splitData$
\EndFunction
\end{algorithmic}
\end{algorithm}

\begin{algorithm}[H]
\caption{Computing the \frechet gradient}\label{alg:frechetgrad}
\begin{algorithmic}[1]
\Function{frechetGrad}{$AList,aList,A,a$}
\State \emph{// list of networks and a seed network} 
\State Initialize $tanVec$ = [] \Comment{list of tangent vectors}
\State $n$ = number of networks in $AList$
\State $C$ = \Call{optCoup}{$A,B,a,b$}
\For{$j =0,\ldots, n-1$} 
\State $A,a,v, sD$ = \Call{logMap}{$AList[j],aList[j],A,a$}
\State \emph{// $A,a$ may be blown-up at each step} 
\State \parbox{0.85\linewidth}{{Use $sD$ to blow-up rows of $tanVec$ elements to be compatible with the newly blown-up $A$}}
\State Append $v$ to $tanVec$ 
\EndFor
\State $g = sum(tanVec)/n$ \Comment{\frechet gradient}
\State \Return $g$
\EndFunction
\end{algorithmic}
\end{algorithm}

\end{document}